\documentclass{m2an} 

\usepackage{hyperref}
\usepackage{tikz}
\tikzset{draw half paths/.style 2 args={%
  decoration={show path construction,
    lineto code={
      \draw [#1] (\tikzinputsegmentfirst) -- 
         ($(\tikzinputsegmentfirst)!0.5!(\tikzinputsegmentlast)$);
      \draw [#2] ($(\tikzinputsegmentfirst)!0.5!(\tikzinputsegmentlast)$)
        -- (\tikzinputsegmentlast);
    }
  }, decorate
}}
\usepackage{stmaryrd}

\newcommand{\N}{\mathbb{N}}
\renewcommand{\P}{\mathbb{P}}
\newcommand{\Q}{\mathbb{Q}}
\newcommand{\R}{\mathbb{R}}

\newcommand{\dd}{\text{d}}

\newcommand{\card}{\text{card}}
\newcommand{\norm}[1]{\left\lVert#1\right\rVert}
\newcommand{\snorm}[1]{\left|#1\right|}

\newcommand{\edge}{\sigma}
\newcommand{\edges}{{\mathcal E}}
\newcommand{\edgesint}{{\mathcal E}_{{\rm int}}}
\newcommand{\edgesext}{{\mathcal E}_{{\rm ext}}}
\newcommand{\mesh}{{\mathcal M}}
\newcommand{\eps}{\varepsilon}

\newcommand{\bfn}{{\boldsymbol n}}
\newcommand{\bfc}{{\boldsymbol c}}
\newcommand{\bfv}{{\boldsymbol v}}
\newcommand{\bfx}{{\boldsymbol x}}

\newcommand{\gradi}{{\boldsymbol \nabla}}

\newcommand{\vfe}{X_\mesh}
\newcommand{\vfez}{X_{\mesh,0}}

\newcommand{\cregcard}{C_\edge}
\newcommand{\cregs}{C_S}
\newcommand{\cregconeI}{C_{\edges,c,1}}
\newcommand{\cregconeII}{C_{\edges,c,2}}
\newcommand{\cregtr}{C_T}

\newcommand{\cregrhoI}{C_{\rho}}
\newcommand{\creginv}{C_{inv}}

\newcommand{\miny}{\underline m_{\widehat y}}
\newcommand{\maxy}{\overline m_{\widehat y}}
\newcommand{\minx}{\underline m_{\widehat x}(\widehat y)}
\newcommand{\maxx}{\overline m_{\widehat x}(\widehat y)}

\SetSymbolFont{stmry}{bold}{U}{stmry}{m}{n}

\begin{document}

\title{Nonparametric face based finite element methods on complex hybrid three-dimensional meshes}
\date{2026}
\author{A. Brunel}
\address{Autorité de Sûreté Nucléaire (ASNR), PSN-RES/SA2I/LIE, F-13115, Saint-Paul-lez-Durance Cedex, France (aubin.brunel@asnr.fr)}
\author{J.-C. Latché$^\dagger$}
\address{Autorité de Sûreté Nucléaire (ASNR), PSN-RES/SA2I, F-13115, Saint-Paul-lez-Durance Cedex, France, deceased in October 2025}
\begin{abstract}
In this work, we propose and analyze low-order nonconforming face-based finite element methods for elliptic problems on complex hybrid three-dimensional meshes, composed of arbitrary polytopal cells.
The focus is on scalar problems, but the approach can be extended in a tensor manner to design generalizations of the classical Crouzeix-Raviart and Rannacher-Turek methods for mixed formulations.
We adopt a nonparametric construction of the local approximation spaces.
We demonstrate that it avoids the degradation of approximation properties commonly induced by parametric mappings on distorted meshes.
Using these approximation spaces, we define finite element schemes for very general cells.
We further demonstrate that, on strongly degraded meshes, an approach inspired by Discontinuous Galerkin methods enables removing consistency errors, even in the presence of potentially curved faces.
Explicit examples of the construction of the polynomial space are given for prismatic (of triangular basis) and pyramidal (of quadrilateral basis) cells.
The resulting schemes are shown theoretically and numerically to be stable and convergent, with optimal convergence rate.
\end{abstract}

\maketitle

In two space dimensions, triangular and quadrilateral cells can be combined naturally to mesh complex geometries. 
Typically, quadrilaterals are used to efficiently cover large regions of the domain, while triangles are introduced only where additional geometric flexibility is required.
This flexibility, however, does not extend straightforwardly to three dimensions. 
Indeed, tetrahedral and hexahedral elements cannot be combined in a conforming manner, as the triangular faces of tetrahedra are geometrically incompatible with the quadrilateral faces of hexahedra. 
Consequently, meshing complex three-dimensional domains with both types of cells requires the introduction of pyramidal or prismatic elements as intermediary bricks \cite{owe-00-for,bau-14-fron}.

This issue is particularly relevant in the context of fluid flow simulations. 
For instance, in \cite{sor-2003-mul}, tetrahedral elements in a given mesh may be merged into hexahedra, triangular prisms, or quadrilateral pyramids to increase the accuracy within the boundary layers.
Another common situation arises when building extruded cells in the flow direction, which naturally introduces prismatic elements \cite{mav-97-uni}. 
Conversely, it may be desirable to locally refine hexahedral elements into tetrahedra \cite{bis-98-tet}.
Consequently, solving a problem in a domain meshed with these types of cells requires defining discrete differential operators on these volumes.

In the fluid flow community, staggered schemes have been successful owing to their intrinsic stability.
Such schemes require face-based methods to discretize the velocity elliptic operator; to do so, finite-elements (FE) methods have been extensively used on non-structured meshes.
The first category consists of conforming vector-valued elements, such as the Raviart-Thomas element \cite{rav-77-mix} and the Brezzi-Douglas-Fortin-Marini element \cite{bre-87-eff}. 
The second category includes nonconforming scalar-valued elements. 
The first element of this type was the Crouzeix-Raviart (CR) element \cite{cro-73-con}, introduced for simplices. 
Another notable example is the Rannacher-Turek (RT) element \cite{ran-92-sim}, which extends the CR element to quadrilateral and hexahedral meshes.

Owing to the popularity of these two elements, many extensions have been proposed in the literature. 
For the RT element, the degrees of freedom can be defined either as the mean values on the faces or the values at the face midpoints, leading to distinct approximation spaces. 
The nonparametric Douglas-Santos-Sheen-Ye element \cite{dou-99-non} seeks to make these two approximation spaces coincide by introducing an additional degree of freedom in the formulation. 
To reduce the number of degrees of freedom to four on quadrilaterals, a parametric counterpart was proposed in \cite{jeo-13-non}.
Another nonparametric quadrilateral element was introduced in \cite{men-18-new}, where shape functions are constructed without solving local linear systems. 
Extensions of the CR element to prismatic cells have also been proposed \cite{ape-01-cro,ape-01-non}, but these approaches are restricted to domains obtained by the extrusion of a two-dimensional surface. 
A generalization to pyramidal elements is presented in \cite{men-19-new}, but it is limited to pyramids with a parallelogram base. 
Finally, \cite{dro-21-non} proposed an extension to general polytopal cells by enriching the finite element space with shape functions defined on dual cells, at the cost of introducing additional degrees of freedom.
Outside the scope of FE methods, closely related approaches on general cells include  the Virtual Element Method \cite{bei-13-bas, can-16-non} and the Hybrid High-Order methods \cite{di-20-hyb}.

However, industrial applications still mostly resort to the CR or the RT elements, owing to their favorable balance between accuracy and computational cost. 
For these reasons, the present work aims to extend these elements for general cells (in particular, quadrilateral pyramidal or triangular prismatic ones), with possibly curved faces, still defining the degrees of freedom solely as mean values on faces. 
Here, we do not choose a parametric approach to construct the approximation space.
Parametric methods, which are widely used in industrial codes owing to their computational efficiency, rely on shape functions defined on a reference element and mapped onto physical cells via invertible geometric transformations. 
Interpolation error estimates are typically obtained using the Bramble-Hilbert lemma \cite{cia-91-han} on the reference element and then transferred to physical cells through these mappings. 
The resulting bounds depend on the norms of the mappings and their inverses, and therefore on some mesh regularity parameters. 
For the RT element, this dependency implies that convergence is guaranteed only if these parameters converge, that is, if distorted cells tend toward regular cells under mesh refinement.
This limitation can be intuitively addressed by enforcing the inclusion of $\P_1$ in the local approximation space.
However, this inclusion is not always possible within a parametric framework while keeping the dimension equal to the number of faces, as shown for hexahedra in \cite{mat-01-map}. 
Moreover, it is a challenging task to build bounded invertible mappings for general pyramidal and prismatic cells \cite{kna-01-inv}.
Therefore, we turn to nonparametric methods, which allow the construction of local approximation spaces directly on physical cells and naturally ensure the inclusion of $\P_1$. 

Using these discrete spaces, we introduce schemes to solve the Poisson problem. 
These schemes are fully general in the sense that we do not assume the types of cells present in the mesh, allowing for composite hybrid meshes. 
We show that the resulting nonconforming finite element scheme achieves optimal convergence rates provided the mesh does not contain curved faces.
In distorted meshes, quadrilateral faces may have four vertices that do not lie in a single plane, resulting in curved faces.
In such cases, the consistency error cannot be properly controlled, and modifications similar to Discontinuous Galerkin formulations \cite{arn-02-uni} are required.
Numerical experiments illustrate the accuracy of the proposed schemes in these different settings. 

Finally, we provide explicit constructions of local approximation spaces, particularly for hexahedral, (triangular) prismatic, and (quadrilateral) pyramidal elements. 
The polynomial bases are chosen to ensure unisolvency, following an approach inspired by \cite{kos-12-new}.
They are constructed using a set of unisolvent polynomials on a reference element and transported to the real element through a linear coordinate transformation. 
This approach differs from parametric methods, where shape functions are mapped to the real elements; here, basis functions are mapped instead.
We emphasize that the coordinate transformation does not need to be the mapping that defines real cells from the reference element. 
It is, in fact, not the case here, and the transformation is actually chosen to be affine.

Although the analysis presented in this article is limited to scalar elliptic operators, the present work lays the foundation for building Stokes elements on general cells.
It is indeed possible to build the velocity approximation space by extending the polynomial spaces introduced here tensorially, and using piecewise constant elements for the pressure, as done for the CR and the RT elements; this extension is postponed to future work.

The remainder of this paper is organized as follows.
First, in Section \ref{Poisson:sec:mesh}, we define the spatial discretization and the continuous and approximation spaces, as well as the specific notations.
In Section \ref{Poisson:sec:FEM}, we introduce the FE scheme for meshes with no curved faces, and the DG-like scheme for meshes with curved faces.
Furthermore, in Section \ref{Poisson:sec:build}, we propose a methodology for building low-order nonconforming FE with face unknowns on different types of cells.
Finally, we present some numerical results in Section \ref{Poisson:sec:num_res}.
We also enclose several Appendixes.
In Appendix \ref{app:mapping}, we recall the mappings available for defining different types of three-dimensional cells.
Then, in Appendix \ref{app:trace}, we adapt the proofs of a trace lemma to general meshes.
\section{Space discretization and nonconforming finite element space}
\label{Poisson:sec:mesh}

\subsection{The mesh}

\subsubsection{Definition}

Let $\Omega$ be an open, bounded subset of $\R^d$ of boundary $\partial \Omega$, with $d \in \{2,3\}$.
We consider $\mesh$ a decomposition of $\Omega$ into open subdomains $K$, called cells, that form a finite family of non-empty convex disjoint subsets of $\Omega$ such that $\overline \Omega = \cup_{K \in \mesh} \overline K$.
The cells $K \in \mesh$ are polygons if $d=2$ or general polyhedron if $d=3$.
We suppose that the cells are built from their vertices through a geometric mapping from a reference element.
More precisely, it means that for each $K \in \mesh$, there exists a regular polygon or polyhedra $\widehat K$ and a mapping $T_K: \R^d \to \R^d$ such that $T_K(\widehat K)=K$, where $\widehat K$ is unique for all the cells of the same nature (that is to say, for the same polygonal or polyhedral type).
We refer to Appendix \ref{app:mapping} for a discussion on this matter.

We denote by $\edges$ the set of edges (if $d=2$) or faces (if $d=3$) of the mesh, which are bounded subsets of surfaces of $\R^d$ delimiting the boundary $\partial K$ of a cell $K \in \mesh$, and that will be called faces in the rest of this work for simplicity.
For a cell $K \in \mesh$, we denote by $\edges(K)$ the faces of $K$ and thus have $\partial K = \cup_{\edge \in \edges(K)} \overline \edge$.
Any face $\edge \in \edges$ is either an external face if it is part of the boundary of $\Omega$, or an internal face if there exists $(K,L)\in \mesh^2$ with $K \neq L$ such that $\overline K \cap \overline L  = \overline \edge$: we denote in this case $\edge = K|L$.
The sets of external faces and internal faces are denoted, respectively, $\edgesext$ and $\edgesint$.

For $K \in \mesh$, we denote by $|K|$ the $d$-dimensional Lebesgue measure of $K$.
The mesh size is defined as $h:=\max_{K \in \mesh} h_K$, where for all $K \in \mesh$, $h_K$ is the diameter of the cell $K$.
We also denote by $\rho_K$ the diameter of the largest ball included in $\overline K$, and we denote $\widetilde{B_K}$ this ball.

For a face $\edge \in \edges$, $|\edge|$ is the $(d-1)$-dimensional Lebesgue measure of the face $\edge$.
We also define $h_\edge$ as its diameter and $\bfx_\edge$ as its mass center.
Note that for a cell $K \in \mesh$, a face $\edge \in \edges(K)$ is built as $T_K(\widehat \edge)$, where $\widehat \edge \in \widehat \edges( \widehat K)$ is the corresponding reference edge of the reference cell $\widehat K$.
Furthermore, the restriction of a $d$-dimensional $\mathbb Q_1$ mapping is a $(d-1)$-dimensional $\mathbb Q_1$ mapping, entirely defined by the vertices of the edge.
Note that the mapping is not necessarily affine.
It implies that, when $d=3$, a face $\edge$ might be a generalized quadrilateral, in the sense that its four vertices may not be co-planar.
Such an edge is consequently not planar and is subsequently called curved.
Other types of faces (planar quadrilaterals and triangles when $d=3$, or segments when $d=2$) are called non-curved in the rest of this work.

Finally, for $K \in \mesh$, $\edge \in \edges$ and $\bfx \in \edge$, $\bfn_{K,\edge}(\bfx)$ is the normal vector to the edge $\edge$ at the point $\bfx$ directed outside $K$.
If there exists $L \in \mesh$ such that $\edge = K|L$, then we sometimes use the notation $\bfn_{K,L}(\bfx)$.
If the direction does not matter, we lighten the notation to $\bfn_\edge(\bfx)$.
For a non-curved edge, these vectors are constant along $\edge$; thus, we use the notations $\bfn_{K,\edge}$, $\bfn_{K,L}$  or $\bfn_\edge$, dropping the unneeded dependency on $\bfx$.

\subsubsection{Mesh regularity} \label{subsec:mesh_regularity}

The mesh is assumed to satisfy a few regularity assumptions, given as follows:
\begin{itemize}

    \item[(H1)] There exists a constant $\cregrhoI \in \R_{>0}$ such that, for all $K \in \mesh$
    \begin{align}
         \frac{h_K}{\rho_K} \leq \cregrhoI ; \label{Poisson:eq:cregrhoI}
    \end{align}
    
    \item[(H2)] There exists a constant $\cregs \in \R_{>0}$ such that, for all $K \in \mesh$ and all $\edge \in \edges(K)$
    \begin{align}
         \frac{h_K^{d-1}}{|\edge|} \leq \cregs ; \label{Poisson:eq:cregs}
    \end{align}
    
    \item[(H3)] There exists a constant $\cregcard \in \R_{>0}$ such that, for all $K \in \mesh$
    \begin{align}
         \card(\edges(K)) \leq \cregcard.
         \label{Poisson:eq:cregcard}
    \end{align}
    
\end{itemize}
Hypotheses (H1) and (H2) are standard and related to the distortion of the cells, and hypothesis (H3) holds for the practical applications of this work because the cells considered have at most six faces.

Next, for simplicity, we suppose that the cell is star-shaped with respect to some ball.
This assumption is needed to prove technical results (namely, a trace estimate on curved faces and an interpolation error estimate).
This hypothesis may be slightly relaxed for general meshes, at the cost of additional regularity assumptions and slightly less sharp estimates.

\subsection{Continuous and discrete functional spaces}

Let $\omega$ be an open subset of $\Omega$. 
For $k \in \{0,1,2\}$, we denote by $H^k(\omega)$ the usual Sobolev spaces equipped with the norms $\norm{.}_{H^k(\omega)}$ and semi-norms $\snorm{.}_{H^k(\omega)}$.
For $k=0$, we retrieve the usual Lebesgue space, with the classical norm $\norm{.}_{L^2(\omega)}$.
Moreover, we denote by $H^1_0(\Omega)$ the space of functions of $H^1(\Omega)$ that vanish on $\partial \Omega$.

We now define the FE spaces, which are constructed locally by defining on each cell $K \in \mesh$ a triplet $(K,P_K,\Sigma_K)$, where:
\begin{itemize}
    \item $P_K$ is a polynomial space defined on $K$ such that $\P_1(K) \subset P_K \subset \P_2(K)$, that is to say, $P_K$ contains all polynomials that are at least of degree 1 and contains some polynomials that are at most of degree 2;
    \item $\Sigma_K = \{m_{K,\edge}, \edge \in \edges(K) \}$ is the set of degrees of freedom associated with the mean values of a function on $\edge \in \edges(K)$, that is $m_{K,\edge}(v)=\frac{1}{|\edge|}\int_\edge v_{|K}(\bfx) \dd \gamma(\bfx)$ for a function $v$ whose trace is well-defined and integrable on $\edge$.
\end{itemize}
We also suppose that the set of degrees of freedom $\Sigma_K$ is unisolvent with regard to the polynomial space $P_K$.
For additional details on how the basis functions are constructed, we refer to Section \ref{Poisson:sec:build}.

The global FE space is the set $\vfe \subset L^2(\Omega)$ of functions $u$ such that:
\begin{align*}
    \left\{
    \begin{aligned}
        & \forall K \in \mesh, \; u_{|K} \in P_K, \\
        & \forall \edge = K|L \in \edgesint, \; m_{K,\edge}(u) = m_{L,\edge}(u).
    \end{aligned}\right.
\end{align*}
For homogeneous Dirichlet boundary conditions, we define the corresponding FE space $\vfez$ as:
\begin{align*}
    \vfez: = \left\{ u \in \vfe; \; \forall K \in \mesh, \; \forall \edge \in \edges(K) \cap \edgesext, \; m_{K,\edge}(u)=0 \right\}.
\end{align*}
In the following, for $\edge \in \edges$, we denote by $\varphi_\edge$ the associated shape function, defined by:
\begin{align} \label{cond:basis_func}
    \left\{
    \begin{aligned}
    & \forall K \in \mesh \text{ such that } \edge \in \edges(K), \; \forall \edge' \in \edges(K), \;  m_{K,\edge'}(\varphi_\edge)=\delta_{\edge,\edge'}, \\
    & \forall K \in \mesh \text{ such that } \edge \not \in \edges(K), \; \varphi_{\edge|K} \equiv 0,
    \end{aligned}\right.
\end{align}
where the Kronecker symbol $\delta_{\edge,\edge'}$ means $\delta_{\edge,\edge'}=1$ if $\edge = \edge'$ and $\delta_{\edge,\edge'}=0$ otherwise.
For $K \in \mesh$ and $\edge \in \edges(K)$, the restriction of the shape function $\varphi_\edge$ on $K \in \mesh$ is computed by solving the linear system given by the first part of Condition \eqref{cond:basis_func}.
For a face $\edge \in \edges$, we sometimes write the degrees of freedom associated with $\edge$ simply as $m_\edge$.
On $\vfe$, this notation is indeed unambiguously defined on $\edgesext$, as well as on $\edgesint$ due to the required interfacial continuity ($\ie\; m_{K,\edge}(u)=m_{L,\edge}(u)$ for an internal face $\edge = K|L$ and $u \in \vfe$).
However, it also means that the space is nonconforming in the sense that $\vfe \not \subset H^1(\Omega)$.
Indeed, on a face $\edge \in \edgesint$, $\edge = K|L$, only the interfacial continuity $m_{K,\edge}(u) = m_{L,\edge}(u)$ is enforced for $u \in \vfe$.
In general cases, the function is, however, not continuous at each point of the face $\edge$.

Finally, for a given mesh $\mesh$, we define the broken Sobolev spaces $H^k(\mesh)$ for $k \in \{0,1,2\}$ as the space of functions that belong locally to $H^1$, that is:
\begin{align*}
    H^k(\mesh):= \left\{ u \in L^2(\Omega); \; \forall K \in \mesh, \; u \in H^k(K)\right\},
\end{align*}
and we have in particular $H^1(\Omega) \cup \vfe \subset H^1(\mesh)$.

\begin{rmrk}
Building the local approximation space using low-order polynomials enables the efficient use of low-order numerical integration methods, and motivates the restriction $P_K \subset \P_2(K)$.
Owing to the dimension of $\P_2(K)$, it also restricts the number of faces per cell to $6$ when $d=2$ and to $10$ when $d=3$.
This constraint is not relevant to the practical applications discussed here, as we only consider quadrilateral pyramidal or triangular prismatic cells. 
However, it might be problematic if a cell with a number of faces exceeding these limits is considered. 
In that case, this condition can be relaxed to $P_K \subset \P_m(K)$, for any $m \geq 2$, without affecting the theoretical analysis of the scheme, except for a modification of Lemma \ref{lem:inv}.
\end{rmrk}

\subsection{Discrete semi-norms and other notations}

Since the FE space is nonconforming, we define Sobolev semi-norms in a broken fashion, that is, for $k \in \{0,1,2\}$ and $u\in H^k(\mesh)$:
\begin{align*}
    \snorm{u}_{k,\mesh}^2:= \sum_{K \in \mesh} \snorm{u}_{H^k(K)}^2.
\end{align*}
Note that for a function $u \in H^k(\Omega)$, we have $\snorm{u}_{k,\mesh} = \snorm{u}_{H^k(\Omega)}$.

For $\edge \in \edges$ and $u \in \vfe \cup C^0(\Omega)$, we denote by $[u]_\edge$ the jump of the function $u$ across $\edge$:
\begin{align*}
    & [u]_\edge(\bfx) = u_{|K}(\bfx) - u_{|L}(\bfx), \quad &&\text{for } \edge \in \edgesint, \edge = K|L \text{ such that } \bfn_\edge = \bfn_{K,L}, \text{ and } \bfx \in \edge, \\
    & [u]_\edge(\bfx) = u_{|K}(\bfx),  \quad &&\text{for } \edge \in \edgesext, \edge \in \edges(K), \text{ and } \bfx \in \edge,
\end{align*}
and by $\{u\}_\edge$ the mean value of $u$ on the face $\edge$:
\begin{align*}
    & \{u\}_\edge(\bfx) = \frac{1}{2}(u_{|K}(\bfx) + u_{|L}(\bfx)), \quad &&\text{for } \edge \in \edgesint, \edge = K|L, \text{ and }  \bfx \in \edge, \\
    & \{u\}_\edge(\bfx) = u_{|K}(\bfx),  \quad &&\text{for } \edge \in \edgesext, \edge \in \edges(K), \text{ and } \bfx \in \edge.
\end{align*}
Note that, in this work, the direction of the jump of a function over an internal face $\edge \in \edgesint$ (that is to say, the order of the neighbouring cells) does not matter.

\section{Discrete schemes for the Poisson problem}
\label{Poisson:sec:FEM}

We now turn to discrete schemes to solve the Poisson problem, which reads, with homogeneous Dirichlet boundary conditions:
\begin{align*}
    -\Delta u = f & \quad \text{in } \Omega, \\
    u = 0 & \quad \text{on } \partial \Omega.
\end{align*}
It is well-known that, owing to the Lax-Milgram lemma, there exists an unique weak solution $u$ to this problem, that satisfies:
\begin{align} \label{Poisson:eq:weak_Poisson}
    \left\{\begin{array}{l}
        u \in H^1_0(\Omega), \\
        a(u,v) = (f,v)_{L^2(\Omega)}, \, \forall v \in H^1_0(\Omega),
    \end{array}\right.
\end{align}
where the bilinear form $a$ is defined as:
\begin{align*}
    a:\left|\begin{array}{rcl}
        H^1(\Omega) \times H^1(\Omega) &\to & \R \\
        (u,v) & \mapsto &  \int_\Omega \nabla u \cdot \nabla v \dd \bfx
    \end{array}\right.,
\end{align*}
and $(\cdot,\cdot)_{L^2(\Omega)}$ is the usual $L^2(\Omega)$ scalar product.

We present and analyze two types of schemes to solve this problem.
First, we consider a FE scheme, whose consistency error may be estimated whenever the mesh does not contain any curved faces.
However, this estimate is no longer valid when curved faces are present in the mesh, as the corresponding consistency terms cannot be properly controlled anymore.
In such cases, we introduce additional consistency and stabilization terms, leading to a formulation akin to Discontinuous Galerkin methods.

\subsection{A FE scheme for meshes with non-curved faces}

The FE scheme for the Poisson problem on meshes with non-curved faces is given as:
\begin{align} \label{Poisson:eq:disc_Poisson}
    \left\{\begin{aligned}
        & u_\mesh \in \vfez, \\
        & a_\mesh(u_\mesh,v_\mesh) = (f,v_\mesh)_{L^2(\Omega^d)}, \, \forall v_\mesh \in \vfez,
    \end{aligned}\right.
\end{align}
where the bilinear form $a_\mesh$ is defined as:
\begin{align*}
    a_\mesh:\left|\begin{array}{rcl}
        H^1(\mesh) \times H^1(\mesh) &\to & \R \\
        (u_\mesh,v_\mesh) & \mapsto &  \sum\limits_{K \in \mesh} \int_K \nabla u_\mesh \cdot \nabla v_\mesh
    \end{array}\right.,
\end{align*}
and $(\cdot,\cdot)_{L^2(\Omega^d)}$ is the usual scalar product defined on $L^2(\Omega^d)$.

For $v_\mesh \in \vfe$, we have $a_\mesh(v_\mesh, v_\mesh) = \snorm{v_\mesh}_{1,\mesh}^2$ so that the bilinear form $a_\mesh$ is coercive on $\vfe$.
Thus, the existence and uniqueness of the solution to the discrete Poisson problem \eqref{Poisson:eq:disc_Poisson} follows from the Lax-Milgram theorem and the following lemma, whose proof is recalled for completeness.
\begin{lmm}
    The application $v_\mesh \mapsto \snorm{v_\mesh}_{1,\mesh}$ is a norm on $\vfez$.
\end{lmm}
\begin{proof}
    The sub-additivity and absolute homogeneity are direct. 
    Let us prove the positive definiteness. 
    Suppose that $v_\mesh$ is such that $\snorm{v_\mesh}_{1,\mesh}=0$.
    Then, $(\gradi v_{\mesh})_{|K} = 0$ for all $K \in \mesh$, so that $v_\mesh$ is constant per cell.
    Thus, there exists $(c_K)_{K \in \mesh}$ such that $(v_{\mesh})_{|K} = c_K$ for all $K \in \mesh$.
    Let $\edge \in \edgesint, \edge = K|L$.
    There holds $c_K = m_{K,\edge}(v_\mesh) = m_{L,\edge}(v_\mesh) = c_L$, so that $v_\mesh$ is constant on $K \cap L$.
    It then follows that $v_\mesh$ is constant on the whole domain.
    Since $m_\edge(v_\mesh)=0$ for any $\edge \in \edgesext$, $v_\mesh$ is the zero function.
\end{proof}

We show that the presented scheme enjoys optimal convergence rate.
More precisely speaking, we have the following result:

\begin{thrm}[Error estimate for the FE scheme] \label{thm:ell_err_H1}
    Let $u \in H^2(\Omega)$ be the weak solution of the problem \eqref{Poisson:eq:weak_Poisson}.
    Let $u_\mesh \in \vfe$ be the discrete solution of the problem \eqref{Poisson:eq:disc_Poisson}.
    Assume that the mesh regularity assumptions from \ref{subsec:mesh_regularity} hold.
    Then, the following estimates hold:
    \begin{align*}
        & \norm{u-u_{\mesh}}_{L^2(\Omega)} + h \snorm{u-u_{\mesh}}_{1,\mesh} \leq C h^2 \snorm{u}_{H^2(\Omega)},
    \end{align*}
    where $C \in \R_{>0}$ depends only on the regularity of the mesh.
\end{thrm}

Once an interpolation error estimate and a consistency error estimate are obtained, the proof of this theorem follows from the second Strang lemma and a Aubin-Nitsche duality argument (see for instance Chapter 5 of \cite{cia-91-han}).

\begin{rmrk}
The previous inequality may be compared with the one in Theorem 1 of \cite{ran-92-sim}.
While that result is stated for the Stokes problem, the proof for any elliptic problem follows the same line of thought, splitting the global error into interpolation and consistency errors.
The main difference is the constant in front of the solution norm, which scales as $C(h+\sigma_h)h$ rather than $Ch^2$, implying that convergence in the $H^1$ semi-norm is only reached if $\sigma_h$ tends to $0$.
Here, $\sigma_h$ is a mesh regularity parameter that does not vanish with the mesh size; rather, this parameter vanishes only if the mesh sequence is asymptotically regular (as noted in Remark 3 of the same paper).
Note that one would obtain a similar error estimate by employing parametric methods to build finite elements on general meshes, in which case this assumption would not be reasonable.
The scheme presented here is thus more robust in the sense that it guarantees the convergence of a discrete solutions sequences $(u_{\mesh}^{(m)})_{m \in \N}$ associated with a sequence of meshes $(\mesh^{(m)})_{m \in \N}$, provided some mesh regularity parameters remain uniformly bounded with respect to $m \in \N$.
Note that this deficiency of the parametric RT method is solely attributable to the interpolation error estimate. 
Specifically, it arises from the lack of inclusion $\P_1(K) \subset P_K$ for face-based parametric methods on general non-simplicial cells. 
\end{rmrk}

\subsubsection{Interpolation error estimate}

First, we recall the existence of an operator $\Pi: H^1(\Omega) \to \vfe$, which, as stated in \cite{bre-08-mat} Section 4, satisfies the estimate below.
\begin{thrm}[Interpolation error estimate]\label{lem:interpolation}
Assume that the mesh regularity assumptions from \ref{subsec:mesh_regularity} hold.
Then, for $u \in H^2(\Omega)$ and $K \in \mesh$, we have:
\begin{align}
    \norm{u-\Pi u}_{L^2(K)}+h\snorm{u- \Pi u}_{H^1(K)} \leq C_{int} h_K^2 \snorm{u}_{H^2(K)},
\end{align}
where $C_{int} \in \R_{>0}$ depends only on the regularity of the mesh. 
\end{thrm}
By following the steps of \cite{ver-99-err, ver-99-not}, it is possible to trace how the constant $C_{int}$ depends on various mesh regularity parameters.

\subsubsection{Consistency error estimate}

Let us now take a deeper look at the bilinear form $a_\mesh$.
Its definition introduces a consistency error, coming from the non-conformity of the discrete space $\vfe$ with respect to $H^1(\Omega)$.
Indeed, recall that the starting point of the derivation of this discrete second-order differential operator is the divergence formula.
It yields, in particular, for $K \in \mesh$,  $u \in H^2(\Omega) \cap H^1_0(\Omega)$ and $\Psi \in C_c^\infty(\Omega)$:
\begin{align*}
   - \int_K \Delta u(\bfx) \Psi(\bfx) \dd \bfx 
   = \int_K \nabla u (\bfx) \cdot \nabla \Psi (\bfx) \dd \bfx - \sum_{\edge \in \edges(K)} \int_\edge \Psi(\bfx) \nabla u (\bfx) \cdot \bfn_{K,\edge}(\bfx) \dd \gamma (\bfx).
\end{align*}
The first term of the right-hand side of the previous equality yields the definition of the bilinear form $a_\mesh$.
In addition, it also contains a sum on the faces of the cell $K$, which does not appear in the definition of the discrete operator.
Gathering those terms over all cells $K \in \mesh$, we obtain, after a reordering, a sum involving jumps of the function $u$ on the faces $\edge \in \edgesint$, that is:
\begin{align} \label{eq:jump_fe}
    \sum_{K \in \mesh}\sum_{\edge \in \edges(K)} \int_\edge \Psi(\bfx) \nabla u (\bfx) \cdot \bfn_{K,\edge}(\bfx) \dd \gamma (\bfx) = \sum_{\edge \in \edgesint} \bfn_\edge \cdot \int_\edge [\Psi]_\edge(\bfx) \nabla u (\bfx) \dd \gamma (\bfx).
\end{align}
The regularity of the function $\Psi$ implies the cancellation of this sum, since the jumps over the faces vanish.
It would also be the case for a function $\Psi$ in a discrete space conforming with respect to $H^2(\Omega) \cap H^1_0(\Omega)$, or such that the degrees of freedom are defined as the averages (over both sides of a face) of the normal values at the faces.
It is, however, not the case of the space $X_\mesh$, and the finite-element scheme \eqref{Poisson:eq:disc_Poisson} is not consistent.
Nonetheless, we show in the theorem below that this error term can be properly controlled when dealing with meshes that contain only non-curved faces.

In the rest of this section, we drop the dependency on $\bfx$ as well as the differential elements in each integral for clarity.

\begin{thrm} \label{thm:bound}
    Assume that the mesh regularity assumptions from \ref{subsec:mesh_regularity} hold,{ and that the mesh only contains non-curved faces.}
    Then, for $v \in \vfez $ and $u \in H^2(\Omega)$, we have:
    \begin{align*}
        \sum_{K \in \mesh} \sum_{\edge \in \edges(K)} \int_\edge v \nabla u \cdot \bfn_{K,\edge}
        \leq C_{side} h \snorm{u}_{H^2(\Omega)} \snorm{v}_{1,\mesh},
    \end{align*}
    where $C_{side} \in \R_{>0}$ depends only on the regularity of the mesh.
\end{thrm}

Proving this result requires two technical lemmas. 
The first is a trace lemma, which provides bounds for the norms of a function on the boundary of a domain in terms of the norms of the same function defined within the domain.
Specifically, in this context, for $K \in \mesh$ and $\edge \in \edges(K)$, the $L^2(\edge)$ norm of a regular function can be controlled by its $L^2(K)$ and $H^1(K)$ norms.
\begin{lmm}[Trace lemma] \label{lem:trace}
    Assume that the mesh regularity assumptions from \ref{subsec:mesh_regularity} hold.
    Let $K \in \mesh$.
    Then, for all $f \in H^1(K)$, $K \in \mesh$ and $\edge \in \edges(K)$, the following inequality holds:
    \begin{align} \label{eq:trace}
        \norm{f}_{L^2(\edge)}^2 \leq C_T h_K^{-1} \left(\norm{f}_{L^2(K)}^2 + h_K^2 \norm{\nabla f}_{L^2(K)}^2 \right),
    \end{align}
    where $C_T$ depends only on the mesh regularity.
\end{lmm}

\begin{rmrk}
    This type of inequality is well known in the literature, with various proofs relying on slightly different mesh hypotheses.
    It is proved, for instance, in Section 1.2 of \cite{di-20-hyb}, for simplicial cells or polytopal cells whose faces may be subdivided into simplicial subfaces.
    In \cite{can-21-hp}, a proof is provided for curved faces, as long as the cell is star-shaped and respects additional regularity hypotheses.
    However, the star-shaped hypothesis might be relaxed: it suffices to assume that a truncated cone exists inside the cell for each of its faces. 
    We refer to Appendix \ref{app:trace} for more details on this result, whose proof is adapted from the one in \cite{ey-09-con}.
    Note that such a trace lemma for curved faces is not required for the analysis of the present finite element scheme.
    However, it is required if one would like to carry out the analysis of the DG-like scheme that will be introduced in the next section.
\end{rmrk}

Next, we provide an inverse estimate, which allows controlling the Sobolev semi-norm of a polynomial on a cell $K \in \mesh$ by another Sobolev semi-norm on the same cell.
\begin{lmm}[Inverse estimate]\label{lem:inv}
Let $m \geq 1$ and $p \in \P_m(\Omega).$
For all $K \in \mesh$ and all $i \in \{0,m\}$, there holds:
\begin{align} \label{eq:inv}
    \snorm{p}_{H^m(K)} \leq \creginv h_K^{i-m} \snorm{p}_{H^i(K)},
\end{align}
where $C_I \in \R_{>0}$ depends only on $i$, $m$ and on the regularity of the mesh.
\end{lmm}
\begin{proof}
    Since $p \in \P_m(\Omega)$, its $m$-th derivative is constant and we have:
    \begin{align*}
        \snorm{p}^2_{H^m(K)} = \frac{|K|}{|\widetilde{B_K}|} \snorm{p}^2_{H^m(\widetilde{B_K})}.
    \end{align*}
    Considering the mapping $T_{\widetilde{B_K}}: B(0,1) \to \widetilde{B_K}$ defined by $T_{\widetilde{B_K}}(\bfx):= \rho_K \bfx + \bfx_{\widetilde{B_K}}$, then by a direct adaptation of Lemma 12.1 of \cite{ern-21-fin}, there holds:
    \begin{align*}
        \snorm{p}_{H^m(\widetilde{B_K})} \leq \widetilde{C_I} \rho_K^{i-m} \snorm{p}_{H^i(\widetilde{B_K})},
    \end{align*}
    with $\widetilde{C_I}$ that depend only on $i$ and $m$.
    Using $|\widetilde{B_K}| = (4 /3)^{d-2} \pi \rho_K^d$, the inclusion $\widetilde{B_K} \subset K$ and the mesh regularity assumption (H1) concludes the proof.
\end{proof}

Using these two results, we are now in position to prove Theorem \ref{thm:bound}:
\begin{proof}[Proof of Theorem \ref{thm:bound}]
For all $\edge \in \edges$, let ${\bf c}_{\edge}$ be a constant vector.
Since the jumps across the faces of the mesh for a function of $\vfe$ vanish, we have, following the same reordering of the sum as above:
\begin{align*}
    \sum_{K \in \mesh} \sum_{\edge \in \edges(K)} \int_\edge v \nabla u \cdot \bfn_{K,\edge} 
    = & \sum_{\edge \in \edgesint} \bfn_\edge \cdot \int_\edge  [v]_\edge \nabla u \\
    = & \sum_{\edge \in \edgesint} \bfn_\edge \cdot \int_\edge [v]_\edge \left( \nabla u - {\bf c}_{\edge} \right) \\
    = & \sum_{\edge \in \edgesint} \bfn_\edge \cdot \int_\edge [v - m_\edge(v)]_\edge \left( \nabla u - {\bf c}_{\edge} \right) .
\end{align*}
The last equality holds, observing that $[v]_\edge = [v - m_\edge(v)]_\edge$ for any $\edge \in \edges$, since the jump across a face of a constant value vanishes.
Next, for a face $\edge = K|L$, we have, thanks to the Cauchy-Schwarz inequality:
\begin{align*}
    \left( \int_\edge [v - m_\edge(v)]_\edge \left( \nabla u - {\bf c}_{\edge} \right) \right)^2 
    \;\leq \;
    I_v I_{\gradi u},
\end{align*}
where
\begin{align*}
    I_v  \;:=\; \int_\edge [v - m_\edge(v)]_\edge^2, \qquad
    I_{\gradi u} \;:=\;  \int_\edge | \nabla u-  {\bf c}_{\edge} |^2.
\end{align*}
Let us first bound the integral $I_v$.
Splitting the jump of $v - m_\edge(v)$ on each neighbouring cell, and using a Poincaré-Wirtinger Lemma on the face $\edge$, there holds:
\begin{align*}
    I_v 
    \;\leq\; C_P h_\edge^2 \sum_{M = K,L} \snorm{v_{|M}}_{H^1(\edge)}^2,
\end{align*}
where $C_P \in \R_{>0}$ is related to the Poincaré-Wirtinger constant.
Then, as a consequence of the trace inequality \eqref{eq:trace} and the inverse inequality \eqref{eq:inv} we have:
\begin{align*}
    \snorm{v_{|M}}_{H^1(\edge)}^2 
    \; \leq \; \cregtr h_M^{-1} \left( \snorm{v}^2_{H^1(M)} + h_M^2 \snorm{v}^2_{H^2(M)} \right) \; \leq \; \cregtr (1 + \creginv^2) h_M^{-1} \snorm{v}^2_{H^1(M)}.
\end{align*}
Since $h_\edge \leq h_M$, it follows that:
\begin{align*}
    I_v \;\leq\; C_P \cregtr (1+ \creginv^2) \sum_{M = K|L} h_M  \snorm{v}^2_{H^1(M)}.
\end{align*}
We now turn to the integral $I_{\gradi u}$.
On account of the trace inequality \eqref{eq:trace}, there holds, for a cell $K \in \mesh$ such that $\edge \in \edges(K)$: 
\begin{align*}
    I_{\gradi u}
    \;\leq\; \cregtr h_K^{-1} \left( \norm{\nabla u -  {\bf c}_{\edge}}^2_{L^2(K)} + h_K^2 \norm{ \gradi(\nabla u)}^2_{L^2(K) } \right).
\end{align*}
Choosing ${\bf c}_{\edge}$ as the mean of $\nabla u$ on $K$, leads to, thanks to the classical Poincaré-Wirtinger inequality:
\begin{align*}
    I_{\gradi u}
    \; \leq \; \cregtr h_K^{-1} \left( C_P h_K^2 \snorm{u}^2_{H^2(K)} + h_K^2 \snorm{u}^2_{H^2(K) } \right) 
    \; \leq \; \cregtr (1 + C_P) h_K \snorm{u}^2_{H^2(K)},
\end{align*}
where $C_P \in \R_{>0}$ is related to the Poincaré constant.
Since the same bound can be obtained on the other cell $L \in \mesh$ such that $\edge \in \edges(L)$ and $L \neq K$, we can gather the bounds on $I_v$ and $I_{\gradi u}$ to obtain:
\begin{align*}
    \left( \int_\edge \left( [v - m_\edge(v)]_\edge \right) \left( \nabla u - {\bf c}_{\edge} \right) \right)^2
    \leq \cregtr^2 (1+C_P) C_P (1+\creginv^2) \sum_{M = K,L} h_M^2 \snorm{v}^2_{H^1(M)} \snorm{u}^2_{H^2(M)}.
\end{align*}
By the bound of the number of faces of the cells by \eqref{Poisson:eq:cregcard} of Hypothesis (H3), we get:
\begin{align*}
    \sum_{K \in \mesh} \sum_{\edge \in \edges(K)} \int_\edge v \nabla u \cdot \bfn_{K,\edge} 
    \;\leq\; C_{side} \sum_{K \in \mesh} h_K \snorm{u}_{H^2(K)} \snorm{v}_{H^1(K)},
\end{align*}
where $C_{side} = \cregcard \sqrt{ \cregtr^2 (1+C_P) C_P (1+\creginv^2) }$.
The desired result then follows directly from a Cauchy-Schwarz inequality.
\end{proof}

\subsection{A DG-like scheme for meshes with curved faces }
\label{Poisson:sec:DG}

We now suppose, in the rest of this section, that there exist some curved faces in the mesh, that is, there exist some $\edge \in \edges$ where $\bfn_\edge(\bfx)$ is not constant along $\edge$.
Then, for some functions $v \in \vfez$ and $u \in H^2(\Omega)$, one may only write:
\begin{align} \label{eq:jump_dg}
    \sum_{K \in \mesh} \sum_{\edge \in \edges(K)} \int_\edge v(\bfx) \nabla u(\bfx) \cdot \bfn_{K,\edge}(\bfx) \dd \gamma(\bfx)
    = & \sum_{\edge \in \edgesint} \int_\edge  [v]_\edge(\bfx) \nabla u(\bfx)\cdot \bfn_\edge(\bfx) \dd \gamma(\bfx). 
\end{align}
Note that the difference between Equations \eqref{eq:jump_fe} and \eqref{eq:jump_dg} arise from the treatment of the normal vectors in the integral on the right-hand side.
For a curved face $\edge \in \edges$, since the normal vector $\bfn_\edge(\bfx)$ is not constant for all $\bfx \in \edge$, it must be integrated along the faces.
To carry the proof of the bound on the consistency term as previously, one would need the identity $\int_\edge [v]_\edge(\bfx) \bfn_\edge(\bfx) \dd \gamma (\bfx) = 0$ to hold.
It is a much stricter condition than the previously needed requirement $\int_\edge [v]_\edge(\bfx) \dd \gamma (\bfx) = 0$.
In general cases, this property unfortunately does not hold.
It is thus not possible to bound the consistency error in the same fashion as previously.
Furthermore, numerical experiments (see Section \ref{Poisson:sec:num_res}) confirm that the FE scheme degenerates when confronted with such meshes and does not converge when the mesh step vanishes.

A possible way to recover optimal convergence is to introduce consistency and stabilization terms to help get some control of this consistency error, which yields a DG-like scheme.
Taking as a model the right-hand side of the previous identity, we first introduce a consistency term in the shape of the bilinear form $c_\mesh$, defined as:
\begin{align*}
   c_\mesh:\left|\begin{array}{rcl}
        H^1(\mesh) \times H^2(\mesh) &\to & \R \\
        (u_\mesh,v_\mesh) & \mapsto &  -\sum\limits_{\edge \in \edges} \int_\edge [v_\mesh]_\edge \{\nabla u_\mesh\}_\edge \cdot \bfn_{K,\edge}
    \end{array}\right..
\end{align*}
Remark that for $u_{\mesh} \in \vfe$, $\nabla u_\mesh(\bfx)$ is not defined along any $\edge \in \edgesint$, so that the bilinear form is defined as the mean value across the face $\{ \nabla u_\mesh \}_\edge$.
Adding this term to the bilinear form $a_\mesh$ leads to a consistent scheme, but yields a non-symmetrical linear system, since the bilinear form $c_\mesh$ is itself non-symmetrical.
To recover symmetry, we also add a symmetrized version of this term.
Then, to recover coercivity, it is required to add a stabilization term $s_\mesh$, which in our case takes the following form:
\begin{align*}
    s_\mesh:\left|\begin{array}{rcl}
        H^1(\mesh) \times H^1(\mesh) &\to & \R \\
        (u_\mesh,v_\mesh) & \mapsto &  \sum\limits_{\edge \in \edgesint} \frac{\gamma_\edge}{h_\edge}\int_\edge [u_\mesh]_\edge [v_\mesh]_\edge
    \end{array}\right.,
\end{align*}
where $\gamma_\edge \in \R_{>0}$ is a user-fixed parameter.
Note that coercivity is only recovered if for any $\edge \in \edges$, $\gamma_\edge$ is above a lower bound that depends on local mesh regularity.
To simplify the formulation of the problem, one may consider a uniform parameter throughout all the faces, that is, there exists $\gamma \in \R_{>0}$ such that $\gamma_\edge = \gamma$ for all $\edge \in \edges$.

The DG-like formulation of the problem finally reads:
\begin{align} \label{Poisson:eq:DG_Poisson}
    \left\{\begin{aligned}
        & u_\mesh \in \vfez, \\
        & \widetilde a_{\mesh}(u_\mesh,v_\mesh) = (f,v)_{L^2(\Omega^d)}, \, \forall v_\mesh \in \vfez,
    \end{aligned}\right.
\end{align}
where $\widetilde a_{\mesh}$ is the modified symmetrical bilinear defined as:
\begin{align*}
    & \widetilde a_{\mesh}(u_\mesh,v_\mesh) = a_\mesh(u_\mesh, v_\mesh) + c_\mesh(u_\mesh,v_\mesh) + c_\mesh(v_\mesh,u_\mesh) + s_\mesh(u_\mesh,v_\mesh).
\end{align*}

For a function $u \in H^2(\Omega) \cap H^1_0(\Omega)$, the jumps of $u$ and of $\gradi u \cdot \bfn$ over the internal faces vanish by continuity of the function, so that, for $v_\mesh \in \vfe$, there holds $s_\mesh(u,v_\mesh) = c_\mesh(u,v_\mesh)=0$.
It means that if $u \in H^2(\Omega) \cap H^1_0(\Omega)$ is the solution of weak Poisson problem \eqref{Poisson:eq:weak_Poisson}, then, by an integration by parts, there holds:
\begin{align*}
    \widetilde a_{\mesh}(u,v_\mesh) \; = \; a_\mesh(u,v_\mesh) + c_\mesh(u, v_\mesh) \; = \; (f,v_\mesh)_{L^2(\Omega)},
\end{align*}
and the discrete DG-like scheme \eqref{Poisson:eq:DG_Poisson} is consistent.

The optimal order of convergence is recovered using this scheme even when dealing with meshes that contain curved faces, as stated in the following theorem (whose proof may be adapted from the one in \cite{arn-02-uni}): 
\begin{thrm}[Error estimate for the DG-like scheme] \label{thm:ell_err_H1_DG}
    Suppose that $f \in L^2(\Omega)$.
    Let $u \in H^2(\Omega)$ be the weak solution of the problem \eqref{Poisson:eq:weak_Poisson}.
    Let $u_\mesh \in \vfe$ be the discrete solution of the problem \eqref{Poisson:eq:DG_Poisson}.
    Then, the following estimate holds:
    \begin{align*}
        & \norm{u-u_{\mesh}}_{L^2(\Omega)} + h\snorm{u-u_{\mesh}}_{1,\mesh} 
        \leq C h^2 \snorm{u}_{H^2(\Omega)}, 
    \end{align*}
    where $C \in \R_{>0}$ depends only on the regularity of the mesh.
\end{thrm}

\begin{rmrk}
The scheme given here is known as the interior penalty (IP) method \cite{nit-71-var, dou-76-int}, but other DG schemes may be obtained using other consistency or stabilization terms.
We refer to \cite{arn-02-uni} for a more detailed discussion on the different DG methods, their stability condition, and their scope of application.
\end{rmrk}

\section{A methodology to build a finite element on different types of cells}
\label{Poisson:sec:build}

We propose in this section a methodology to build, for any cell $K \in \mesh$, a polynomial space $P_K$ so that: (i) it is consistent with the theoretical framework developed above, namely $\P_1(K) \subset P_K \subset \P_2(K)$ and (ii) it yields an unisolvent FE with a set of face-based degrees of freedom.
A necessary condition to obtain unisolvency is that the dimension of the polynomial space coincides with the number of faces of the cell $K$. 
Since there exists only $d+1$ polynomials of degree 1 or less, it means, for most cell types, that the approximation space must be enlarged with other polynomials.

Note that for each cell type, the set of polynomial sets proposed below is only one choice among many.
Formal computations on a reference configuration indeed show that several choices satisfy the above requirements.
In our case, the selection of the additional polynomials was guided by the numerical conditioning of the associated rigidity matrices.
This property was only verified numerically and is not supported by theoretical claims.

The technique introduced below is inspired by the one proposed in \cite{kos-12-new} for quadrilateral nonparametric FE.
The approach can be divided into two steps.
First, define the polynomial space on a simple reference cell.
Second, obtain the polynomial space on a real cell thanks to a change of coordinates.
More precisely speaking, we proceed as follows:
\begin{itemize}
    \item On a geometrical reference cell $\widehat K$:
    \begin{itemize}
        \item first, define the polynomials $(\widehat p_i)_{0 \leq i \leq d}$ of degree one or less, that is $\widehat p_0(\widehat \bfx) = 1$ and $\widehat p_i( \widehat\bfx) = \widehat{x_i}$ for $1 \leq i \leq d$;
        \item then, for each missing $j$ such as $d+1 \leq j \leq \card(\edges(K))$, define a polynomial $\widehat p_j \in \P_2(\widehat K)$ such that $(\widehat K, P_{\widehat K}, \sigma_{\widehat K})$ is unisolvent, where $P_{\widehat K} =: \{ \widehat p_i, 0 \leq i \leq \card(\edges(K))\}$.
    \end{itemize}  
    \item On the real cell $K$:
    \begin{itemize}
        \item define, from the reference cell, a linear change of coordinates and its matrix $R_K$;
        \item build the polynomial space $P_K:= \text{span}\{ p_i, 0 \leq i \leq \card(\edges(K)) \}$ where the polynomial $p_i$ are defined as $p_i (\bfx) = \widehat p_i \circ R_K^{-1}(\bfx)$.
    \end{itemize}
\end{itemize}
It is important to note that $\P_1(\widehat K) \subset P_{\widehat K}$ and $\text{span}( \P_1(\widehat K) \circ R_K^{-1}) = \text{span}(\P_1(K))$, since the change of coordinates is affine, so that we indeed have $\P_1(K) \subset P_K$.
The space $P_K$ is thus implemented as:
\begin{align*}
    P_K = \text{span}\Big\{\P_1(K)\Big\} \oplus \text{span} \Big\{ \widehat p_j \circ R_{K}^{-1}, \, d+1 \leq j \leq \card (\edges(K)) \Big\}.
\end{align*}
Once the polynomial space $P_K$ is defined, the shape functions $(\varphi_\edge)_{\edge \in \edges(K)}$ can be computed on $K$ by solving the usual linear problem associated with $P_K$ and $(m_{K,\edge})_{\edge \in \edges(K)}$.

Let us give the rationale behind this method.
On a general cell $K \in \mesh$, it is challenging to define in an elementary way an unisolvent FE.
It is, however, always possible to identify a simple polynomial space $P_{\widehat K}$, on a reference cell $\widehat K$ with simple geometry, resulting in an unisolvent reference FE.
Then, one may define a mapping $T_{\widetilde K}$ of the form $T_{\widetilde K}(\widehat \bfx) = R_K \widehat \bfx + \bfc_K$, that maps the reference cell $\widehat K$ to an "affine approximation" $\widetilde K$ of the real cell $K$, akin to the construction given in \cite{aru-95-geo} for quadrilaterals.
Here, $\bfc_K$ is a vector that depends on the real-world cell $K$, and may be chosen, for instance, as the mass center $\bfx_K$ of the cell.
This construction induces an unisolvent FE on $\widetilde K$ by considering the polynomial space $P_K$ together with its associated degrees of freedom.
If $\widetilde K$ is sufficiently close to $K$, then the linear system characterizing unisolvency on $K$ can be viewed as a perturbation of the corresponding system on $\widetilde K$.
It is thus reasonable to expect this system to be invertible as well.
Proving unisolvency on a real cell by formally computing the solution of the linear system is possible; it is, however, a fastidious work that is not pursued here.
Nonetheless, the numerical experiments we provide below testify to the robustness of this approach.

This methodology may be applied to build a nonparametric FE space from existing methods that may not be well-defined on general cells.
As a first example, one may consider the classical RT element on quadrilateral cells.
This element consists of a rotation of 45 degrees of the classical polynomial space $\Q_1$ to obtain the space $\Q_1^{rot}$, given as $\Q_1^{rot}=\{1, x, y, x^2 - y^2\}$, which yields an unisolvent FE on the reference square $[0,1]^2$ with face-based degrees of freedom.
However, the unisolvence of this FE fails if the reference square is itself rotated by 45 degrees.
It can be easily shown that cells having one diagonal parallel to one of the axes shares this defect.
As shown in Figure \ref{Poisson:fig:FEM_Quad}, one solution is to adjust the space $\Q_1^{rot}$ to the cell under consideration.

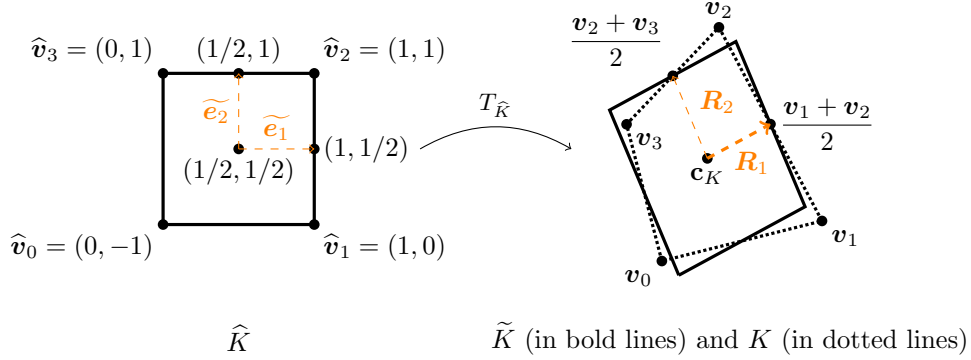
\begin{figure}[hbt!]
    \centering
    \begin{tikzpicture}
        \newcommand{\KRefLeftBottom}{-2,-1}
        \newcommand{\KRefLeftTop}{-2,1}
        \newcommand{\KRefRightBottom}{0,-1}
        \newcommand{\KRefRightTop}{0,1}
        
        \draw [line width=1.2pt] (\KRefLeftBottom) -- (\KRefLeftTop) -- (\KRefRightTop) -- (\KRefRightBottom) -- (\KRefLeftBottom) -- cycle ;
        \foreach \x in {(\KRefLeftBottom), (\KRefLeftTop), (\KRefRightBottom), (\KRefRightTop)}{
        \fill \x circle[radius=2pt];}

        \draw (-1, -2.5) node {$\widehat K$} ;
        \draw (\KRefLeftBottom) node[below left] {
            $\widehat \bfv_0 = (0,-1)$
        };
        \draw (\KRefRightBottom) node[below right] {
            $\widehat \bfv_1 = (1,0)$
        };
        \draw (\KRefRightTop) node[above right] {
            $\widehat \bfv_2 = (1,1)$
        };
        \draw (\KRefLeftTop) node[above left] {
            $\widehat \bfv_3 = (0,1)$
        };
        \draw (-1,0) node[below] {$(1/2,1/2)$} ;
        \draw (-1,1) node[above] {$(1/2,1)$} ;
        \draw (0,0) node[right] {$(1,1/2)$} ;
        
        \foreach \x in {(-1,0), (-1,1), (0,0) }{
        \fill \x circle[radius=2pt];}

        \draw [dashed, ->, orange] (-1,0) -- (0,0) node[midway, above] {$\widetilde{\boldsymbol e_1}$};
        \draw [dashed, ->,orange] (-1,0) -- (-1,1)  node[midway, left] {$\widetilde{\boldsymbol e_2}$};
        
        \path[->, every node/.style={font=\sffamily\small}]
         (1.4,0) edge[bend left] node [above] {$T_{\widehat K}$} (3.4,0);
        
        \newcommand{\RealLeftBottom}{4.598269515458673, -1.4827722283113838}
        \newcommand{\RealLeftTop}{4.131282217350893, 0.3260734985006058}
        \newcommand{\RealRightBottom}{6.716922863405994,-0.9523874705820474}
        \newcommand{\RealRightTop}{5.353525403784439,1.6090862003928252}
        
        \newcommand{\AffLeftBottom}{4.822372055837117,-1.6709292143521046}
        \newcommand{\AffRightBottom}{6.492820323027551,-0.7642304845413266}
        \newcommand{\AffRightTop}{5.577627944162883,1.4209292143521042}
        \newcommand{\AffLeftTop}{3.9071796769724494,0.5142304845413266}
        
        \draw [densely dotted,line width=1.2pt] (\RealLeftBottom) -- (\RealLeftTop) -- (\RealRightTop) -- (\RealRightBottom) -- (\RealLeftBottom) ;
        
        \draw [line width=1.2pt] (\AffLeftBottom) -- (\AffLeftTop) -- (\AffRightTop) -- (\AffRightBottom) -- (\AffLeftBottom) ;
                
        \foreach \x in {(\RealLeftBottom), (\RealLeftTop), (\RealRightBottom), (\RealRightTop)}{
        \fill \x circle[radius=2pt];}    
        
        \foreach \x in {(5.2, -0.125), (6.035224133595217,0.3283493649053891), (4.742403810567666,0.9675798494467152)}{
        \fill \x circle[radius=2pt];}   
        
        \draw (5.5,-2.5) node {$\widetilde K$ (in bold lines) and $K$ (in dotted lines) } ;
        \draw (\RealLeftBottom)     node[below left] {$\bfv_0$} ;
        \draw (\RealLeftTop)        node[below right] {$\bfv_3$} ;
        \draw (\RealRightBottom)    node[below right] {$\bfv_1$} ;
        \draw (\RealRightTop)       node[above] {$\bfv_2$} ;
        \draw (5.2, -0.125) node[below] {${\bf c}_K$} ;
        \draw (6.035224133595217,0.3283493649053891) node[right] {$\dfrac{\bfv_1 + \bfv_2}{2}$} ;
        \draw (4.742403810567666,0.9675798494467152) node[above left] {$\dfrac{\bfv_2 + \bfv_3}{2}$} ;

        \draw [dashed, line width=1.2pt, ->, orange] (5.2, -0.125) -- (6.035224133595217,0.3283493649053891)  node[midway, below right, xshift=-0.2cm] {$\boldsymbol R_1$};
        \draw [dashed, ->, orange] (5.2, -0.125) -- (4.742403810567666,0.9675798494467152)  node[midway, right, yshift=0.2cm] {$\boldsymbol R_2$};
        
    \end{tikzpicture}
    \caption{The change of coordinates $T_{\widehat K}$ is defined so that the vector $\widetilde{\boldsymbol e_1}:= (1/2,0)$ is sent on $\boldsymbol R_1 = \frac{\bfv_1 + \bfv_2}{2} - \bfc_K$ and the vector $\widetilde{\boldsymbol e_2}:= (0,1/2)$ is sent on $\boldsymbol R_2 = \frac{\bfv_2 + \bfv_3}{2} - \bfc_K$, where $\bfc_K$ is the mass center of the cell $K$. 
    In particular, the space $P_K$ yields an unisolvent FE on $\widetilde K$, which is the (affine) parallelogram approximation of $K$ that shares its faces midpoints. }
    \label{Poisson:fig:FEM_Quad}
\end{figure}

For a set of different natures of cells, we give some sensible choices for the reference cell $\widehat K$, the polynomials $(\widehat p_j)_{d+1 \leq j \leq \card(\edges(K))}$ and the changes of coordinates.
We emphasize that these choices are not unique; in particular, the set of polynomials might be obtained by trial and error, aiming for unisolvency on the reference cell.
Furthermore, this approach is not limited to the types of cells considered in this work.

\begin{rmrk}[Difference with parametric methods]
At first glance, the difference between the approach presented here and parametric methods may not be apparent.
Both methods, in fact, rely on mapping polynomials from a reference cell.
In the case of parametric methods, a set of basis functions is first defined on this reference cell.
Then, shape functions are determined on the same reference cell using the chosen degrees of freedom and are combined with the aforementioned mapping to build the shape functions on the physical cell.

In the present approach, a set of basis functions is also defined on a reference cell.
Shape functions are, in contrast, not computed on the reference cell; rather, the basis functions are directly mapped on the real cell, where the shape functions are then constructed.
Consequently, linear systems must be solved for the latter on each real cell, rather than only once on the reference cell in the case of (iso-)parametric method.
The obvious drawback is an increase in the numerical cost of the scheme.
However, because the size of each of the linear systems is equal to the number of faces of the considered cell, the additional cost is, in most cases, negligible.
It is especially true when using this method to solve time-dependent problems.
In such cases, the computed shape functions can be stored in memory so that the supplement price is paid only once, at the initial time step.
Furthermore, the computation of the basis functions is completely independent across cells and is thus embarrassingly parallel.
\end{rmrk}

\subsection{Hexahedral case}

On hexahedra, we may take the polynomial space of the Rannacher-Turek element:
\begin{align*}
    P_{\widehat K}:=\text{span}\left\{ 1, x, y, z, x^2 - y^2, y^2 - z^2 \right\}, 
\end{align*}
which yields an unisolvent FE on the reference cell $\widehat K = [0,1]^3$ of vertices $\hat \bfv_0 = (0,0,0)$, $\hat \bfv_1 = (0,1,0)$, $\hat \bfv_2 = (0,0,1)$, $\hat \bfv_3 = (1,0,0)$, $\hat \bfv_4 = (1,1,0)$, $\hat \bfv_5 = (1,0,1)$, $\hat \bfv_6 = (0,1,1)$ and $\hat \bfv_7 = (1,1,1)$.
For a cell $K$ whose vertices are ordered in a similar way, the change of coordinates is defined by the matrix $R_K$ whose columns are $\boldsymbol R_{K,1}, \boldsymbol R_{K,2}$ and $\boldsymbol R_{K,3}$ and are defined by:
\begin{align*}
    \bfc_K & = \frac{1}{8} \sum_{i = 0}^7 \bfv_i, &
    \boldsymbol R_{K,1} & = \frac{1}{4} ( \bfv_3 + \bfv_4 + \bfv_5 + \bfv_7 ) - \bfc_K, & \\
    \boldsymbol R_{K,2} & = \frac{1}{4} ( \bfv_1 + \bfv_4 + \bfv_6 + \bfv_7 )-\bfc_K, &
    \boldsymbol R_{K,3} & = \frac{1}{4} ( \bfv_2 + \bfv_5 + \bfv_6 + \bfv_7 )-\bfc_K. &
\end{align*}

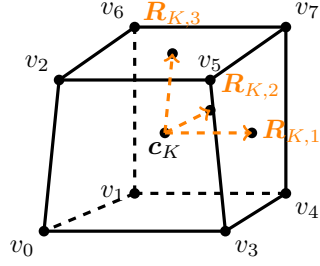
\begin{figure}[htb!]
    \centering
    \begin{tikzpicture}
    \newcommand{\vzero}{-0.2,0}
    \newcommand{\vthree}{2.2,0}
    \newcommand{\vfour}{3,0.5}
    \newcommand{\vone}{1,0.5}
    \newcommand{\vtwo}{0,2}
    \newcommand{\vfive}{2,2}
    \newcommand{\vseven}{3,2.7}
    \newcommand{\vsix}{1,2.7}

    \newcommand{\cK}{1.4,1.3}
    \newcommand{\MidRight}{2.55,1.3}
    \newcommand{\MidBack}{2,1.6}
    \newcommand{\MidTop}{1.5,2.35}

    \draw[line width=1.2pt] (\vzero) -- (\vthree);
    \draw[line width=1.2pt] (\vthree) -- (\vfour);
    \draw[dashed, line width=1.2pt] (\vfour) -- (\vone);
    \draw[line width=1.2pt, dashed] (\vone) -- (\vzero);

    \draw[line width=1.2pt] (\vtwo) -- (\vfive);
    \draw[line width=1.2pt] (\vfive) -- (\vseven);
    \draw[line width=1.2pt] (\vseven) -- (\vsix);
    \draw[line width=1.2pt] (\vsix) -- (\vtwo);

    \draw[line width=1.2pt] (\vzero) -- (\vtwo);
    \draw[line width=1.2pt] (\vthree) -- (\vfive);
    \draw[line width=1.2pt] (\vfour) -- (\vseven);
    \draw[dashed, line width=1.2pt] (\vone) -- (\vsix);

    \foreach \x in {
        (\vzero), (\vthree), (\vfour), (\vone), (\vtwo), (\vfive), (\vseven), (\vsix),
        (\cK), (\MidRight), (\MidBack), (\MidTop)
    }{
        \fill \x circle[radius=2pt];
    }

    \draw (\vzero) node[below left] {$v_0$};
    \draw (\vthree) node[below right] {$v_3$};
    \draw (\vfour) node[below right] {$v_4$};
    \draw (\vone) node[left] {$v_1$};
    \draw (\vtwo) node[above left] {$v_2$};
    \draw (\vfive) node[above] {$v_5$};
    \draw (\vseven) node[above right] {$v_7$};
    \draw (\vsix) node[above left] {$v_6$};

    \draw (\cK) node[below] {$\bfc_K$};

    \draw[dashed, line width=1.2pt, ->, orange] (\cK) -- (\MidRight)
        node[right] {$\boldsymbol R_{K,1}$};
    \draw[dashed, line width=1.2pt, ->, orange] (\cK) -- (\MidBack)
        node[above right] {$\boldsymbol R_{K,2}$};
    \draw[dashed, line width=1.2pt, ->, orange] (\cK) -- (\MidTop)
        node[above, yshift = 0.2cm] {$\boldsymbol R_{K,3}$};
\end{tikzpicture}
    \caption{An hexahedral cell $K$, with its vertex and the vectors used to define the linear change of coordinate}
    \label{fig:transfo_hex}
\end{figure}

\subsection{Quadrilateral pyramidal case}

The reference polynomial span for quadrilateral pyramidal cells is chosen as follows:
\begin{align*}
    \widehat P_K:=\text{span}\left\{ 1, x, y, z, x^2 - y^2\right\},
\end{align*}
which yields an unisolvent FE on the reference pyramid $\widehat K$ whose basis vertices are $\hat \bfv_0 = (0,0,0)$, $\hat \bfv_1 =(1,0,0)$, $\hat \bfv_2 =(1,1,0)$ and $\hat \bfv_3 =(0,1,0)$ and the vertex of its apex is $\hat \bfv_4 =(0,0,1)$.
For a cell $K$ of vertices $(\bfv_i)_{0 \leq i \leq 4}$, the change of coordinates is defined by the matrix $R_K$ whose columns are $\boldsymbol R_{K,1}, \boldsymbol R_{K,2}$ and $\boldsymbol R_{K,3}$ and are defined by:
\begin{align*}
    \bfc_K & = \frac{1}{4} \sum_{i = 0}^3 \bfv_i, & 
    \boldsymbol R_{K,1} & = \frac{1}{2} ( \bfv_1 + \bfv_2 ) - \bfc_K, &
    \\
    \boldsymbol R_{K,2} & = \frac{1}{2} ( \bfv_2 + \bfv_3 ) - \bfc_K , &
    \boldsymbol R_{K,3} & = \bfv_4 - \bfc_K. &
\end{align*}

\begin{figure}[hbt!]
    \centering
    \begin{tikzpicture}
    \newcommand{\LeftBack}{1.2,1}
    \newcommand{\LeftFront}{0,0.5}
    \newcommand{\RightBack}{4.8,1}
    \newcommand{\RightFront}{3,0}
    \newcommand{\Top}{2,3}

    \newcommand{\Mid}{2.25,0.625}
    \newcommand{\MidRight}{3.9,0.5}
    \newcommand{\MidBack}{3,1}

    \draw[line width=1.2pt] (\LeftFront) -- (\RightFront) -- (\RightBack) ;
    \draw [dashed, line width=1.2pt](\RightBack) -- (\LeftBack) -- (\LeftFront) ;
    \draw[line width=1.2pt] (\Top) -- (\LeftFront) ;
    \draw[line width=1.2pt] (\Top) -- (\RightFront) ;
    \draw[line width=1.2pt] (\Top) -- (\RightBack) ;
    \draw [dashed, line width=1.2pt](\Top) -- (\LeftBack) ;

    \foreach \x in {
        (\LeftBack), (\LeftFront), (\RightBack), (\RightFront), (\Top), (\Mid), (\MidRight), (\MidBack)
    }{
        \fill \x circle[radius=2pt];
    }

    \draw (\LeftFront) node[below left] {$v_0$};
    \draw (\RightFront) node[below] {$v_1$};
    \draw (\RightBack) node[right] {$v_2$};
    \draw (\LeftBack) node[below] {$v_3$};
    \draw (\Top) node[above] {$v_4$};
    \draw (\Mid) node[below] {$\bfc_K$} ;

    \draw[dashed, line width=1.2pt, ->, orange] (\Mid) -- (\MidRight)
        node[below right] {$\boldsymbol R_{K,1}$};
    \draw[dashed, line width=1.2pt, ->, orange] (\Mid) -- (\MidBack)
        node[above right] {$\boldsymbol R_{K,2}$};
    \draw[dashed, line width=1.2pt, ->, orange] (\Mid) -- (\Top)
        node[above right] {$\boldsymbol R_{K,3}$};
    
    \end{tikzpicture}
    \caption{A quadrilateral pyramidal cell $K$, with its vertex and the vectors used to define the linear change of coordinate}
    \label{fig:transfo_pyr}
\end{figure}
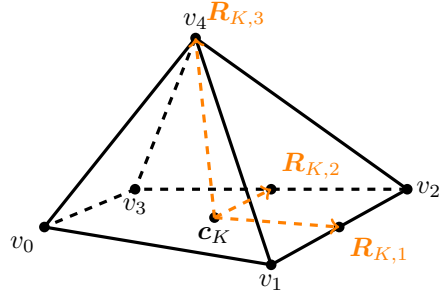

\subsection{Triangular prismatic case}

For triangular prismatic cells, one may take $P_{\widehat K}$ as follows:
\begin{align*}
    P_{\widehat K}:=\text{span}\left\{ 1, x, y, z, x^2 - y^2\right\},
\end{align*}
which yields an unisolvent FE on the reference prism $\widehat K$ whose vertices are $\hat \bfv_0 = (0,0,0)$, $\hat \bfv_1 = (0,1,0)$, $\hat \bfv_2 = (0,0,1)$, $\hat \bfv_3 = (1,0,0)$, $\hat \bfv_4 = (1,1,0)$ and $\hat \bfv_5 = (1,0,1)$.
For a cell $K$ of vertices $(\bfv_i)_{0 \leq i \leq 5}$, the change of coordinates is defined by the matrix $R_K$ whose columns are $\boldsymbol R_{K,1}, \boldsymbol R_{K,2}$ and $\boldsymbol R_{K,3}$ and are defined by:
\begin{align*}
    \bfc_K & = \frac{1}{2} ( \bfv_0 + \bfv_3 ), & 
    \boldsymbol R_{K,1} & = \bfv_3 - \bfc_K, &
    \\
    \boldsymbol R_{K,2}  & = \frac{1}{2} ( \bfv_1 + \bfv_4 ) - \bfc_K, & \quad 
    \boldsymbol R_{K,3} & = \frac{1}{2} ( \bfv_2 + \bfv_5 ) - \bfc_K. &
\end{align*}

\begin{figure}[hbt!]
    \centering
    \begin{tikzpicture}
    \newcommand{\TriLeftBack}{1.2,1}
    \newcommand{\TriLeftFront}{0,0}
    \newcommand{\TriLeftTop}{-0.5,2.6}
    \newcommand{\TriRightBack}{4.2,1}
    \newcommand{\TriRightFront}{2.,0}
    \newcommand{\TriRightTop}{3,3}
    \newcommand{\MidBack}{2.7,1}
    \newcommand{\MidFront}{1.,0}
    \newcommand{\MidTop}{1.25,2.8}

    \draw[line width=1.2pt] (\TriLeftTop) -- (\TriLeftFront) ;
    \draw [line width=1.2pt, dashed](\TriLeftFront) -- (\TriLeftBack) -- (\TriLeftTop) ;
    
    \draw[line width=1.2pt] (\TriRightFront) -- (\TriRightBack) -- (\TriRightTop) -- (\TriRightFront) ;
    
    \draw[line width=1.2pt] (\TriLeftFront) -- (\TriRightFront) ;
    \draw[line width=1.2pt] (\TriLeftTop) -- (\TriRightTop) ;
    \draw [dashed, line width=1.2pt](\TriLeftBack) -- (\TriRightBack) ;

    \foreach \x in {
        (\TriLeftBack), (\TriLeftFront), (\TriLeftTop), (\TriRightBack), (\TriRightFront), (\TriRightTop), (\MidBack), (\MidFront), (\MidTop)
    }{
        \fill \x circle[radius=2pt];
    }

    \draw (\TriLeftFront) node[below left] {$\bfv_0$};
    \draw (\TriLeftBack) node[below right] {$\bfv_1$};
    \draw (\TriLeftTop) node[above left] {$\bfv_2$};
    \draw (\TriRightFront) node[below right] {$\bfv_3$};
    \draw (\TriRightBack) node[below right] {$\bfv_4$};
    \draw (\TriRightTop) node[above] {$\bfv_5$};
    \draw (\MidFront) node[below left] {$\bfc_K$} ;

    \draw[dashed, line width=1.2pt, ->, orange] (\MidFront) -- (\TriRightFront)
        node[midway, below] {$\boldsymbol R_{K,1}$};
    \draw[dashed, line width=1.2pt, ->, orange] (\MidFront) -- (\MidBack)
        node[above right] {$\boldsymbol R_{K,2}$};
    \draw[dashed, line width=1.2pt, ->, orange] (\MidFront) -- (\MidTop)
        node[midway, right] {$\boldsymbol R_{K,3}$};
    \end{tikzpicture}
    \caption{A triangular prismatic cell $K$, with its vertex and the vectors used to define the linear change of coordinate}
    \label{fig:transfo_prism}
\end{figure}
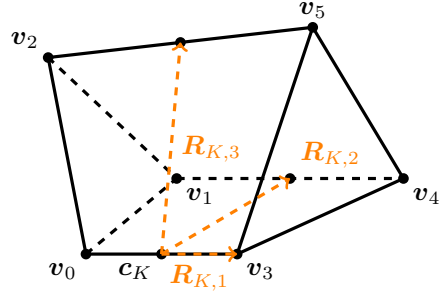

\begin{rmrk}[Parametric element for quadrilateral pyramid and triangular prismatic cells]
We have introduced here a method to construct nonparametric elements, notably on triangular prisms and quadrilateral pyramids.
Still, it is possible to derive parametric counterparts for these types of cells in the following manner.
Let us consider a prismatic or pyramidal cell $K \in \mesh$.
First, build on the corresponding reference element $\widetilde{K}$ shape functions related to the polynomial space $P_{\widetilde{K}}$.
Then, build the shape functions on the real cell $K$ using the mappings given in Appendix \ref{app:mapping}.
Note that the analysis given in Section \ref{Poisson:sec:FEM} still holds, as long as $\P_1(K)$ is contained in the local approximation spaces for all $K \in \mesh$.
This condition, however, only holds if the mappings are affine, which severely restricts the set of admissible cells.
As explained in Appendix \ref{app:mapping}, only quadrilateral pyramids whose basis is a parallelogram can be built using an affine mapping.
Regarding prisms, it means that the triangular faces must lie in parallel planes; in other words, it means that only extruded triangles are valid.
If the mapping toward a cell $K \in \mesh$ is no longer required to be affine, then $\P_1(K)$ is no longer contained in the local approximation space.
Consequently, degradations akin to those observed for the parametric RT element is observed theoretically and numerically.
\end{rmrk}

\section{Numerical results}
\label{Poisson:sec:num_res}

We present in this section some computations performed with the open-source software CALIF$^3$S \cite{califs} developed at ASNR. 
We propose a convergence test for the Poisson problem \eqref{Poisson:eq:weak_Poisson}, on different sequences of meshes.
The domain is $\Omega = [0,1]^3$ and the continuous solution is given as:
\begin{align*}
    u(x,y,z) = \sin(\pi x) \sin(\pi y) \sin(\pi z).
\end{align*}
The right-hand side $f$ of the problem is defined as $f := - \Delta u$, and we enforce homogeneous Dirichlet boundary conditions.

\subsection{Test on meshes with non-curved faces}

First, we propose a test on sequences of regular and distorted meshes with non-curved faces consisting of either hexahedra, triangular prisms, or quadrilateral pyramids, used to illustrate the convergence property of the classical FE scheme \eqref{Poisson:eq:disc_Poisson}.
We build these meshes in three steps:
\begin{itemize}
    \item Firstly, we build a Cartesian two-dimensional grid of $(2^n)^2$ cells, where $n$ is an integer that varies in $\llbracket 3,7 \rrbracket$ with the mesh refinement;
    \item Then, if distorted meshes are desired, a distortion is applied to the two-dimensional grid, which consists in moving each internal vertex in a random direction, with a displacement length equal to $0.2\,h$.
    If not, nothing is done at this step.
    \item Finally, the last step depends on the type of cells:
    \begin{itemize}
        \item For hexahedral meshes, the quadrilateral mesh is extruded in three dimensions into $2^n$ planes, to obtain hexahedra;
        \item For pyramidal meshes, the quadrilateral grid is extruded in three dimensions into $2^n$ planes, to obtain hexahedra. 
        Each hexahedron is then split into six quadrilateral pyramids whose bases are the faces of the hexahedra, and their apex is the centers of gravity of the hexahedral;
        \item For prismatic meshes, the quadrilateral cells are split into two triangles following one of their diagonals. The obtained triangles are then extruded into prisms into $2^n$ planes.
    \end{itemize}
\end{itemize}

Then, we introduce a sequence of hybrid meshes with non-curved faces, which consists of hexahedra, prisms, and pyramids cells.
The construction is as follows:
\begin{itemize}
    \item First, we divide the domain $[0,1]^2$ into $[0,1/3] \!\times \![0,1] \;\cup \; [1/3, 2/3]\!\times\![0,1] \; \cup\;  [2/3, 1] \!\times\! [0,1]$.
    \item Then, each of the subdomain is meshed with quadrilateral cells, using $n'$ cells in the $x$ direction and $3n'$ in the $y$ direction, where the integer $n' \in \{3,6,12,24,48\}$ varies with the mesh refinement.
    \item Next, each of the subdomain is extruded into $3n'$ layers.
    \item Finally, cells in the extrusion of the first subdomain are left untouched, cells in the extrusion of the second subdomain are refined (in the same fashion as previously) into quadrilateral pyramids, and cells in the extrusion of the last subdomain are refined (once again, as previously) into triangular prisms.
\end{itemize}
Comparisons of extracted sections of the coarsest distorted meshes with non-curved faces and distorted meshes with curved faces (used for the next test) are shown in Figure \ref{fig:meshes}.

\begin{figure}[htb!]
    \centering
    \includegraphics[width=0.24\linewidth]{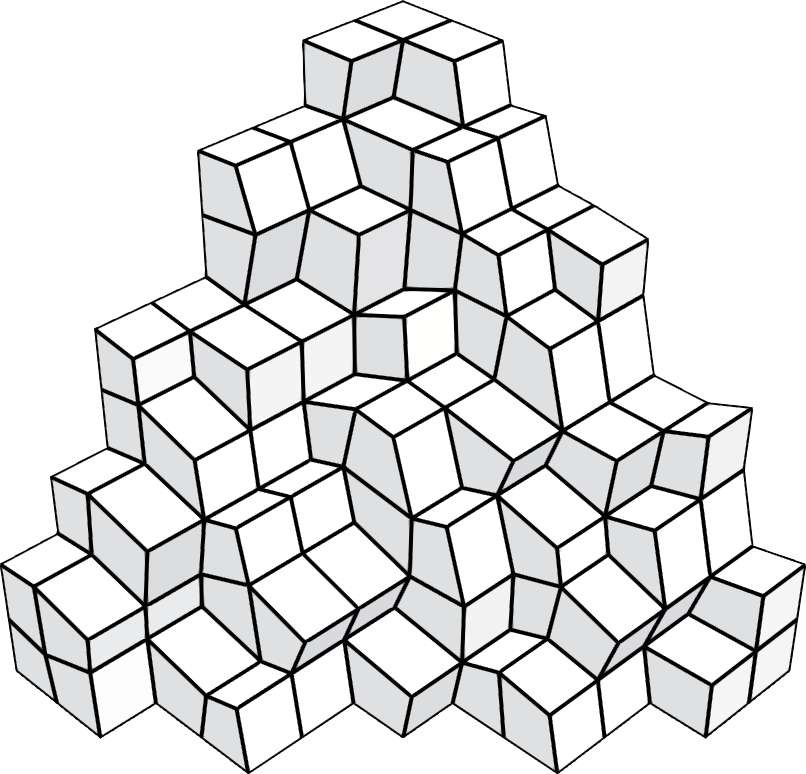}
    \includegraphics[width=0.24\linewidth]{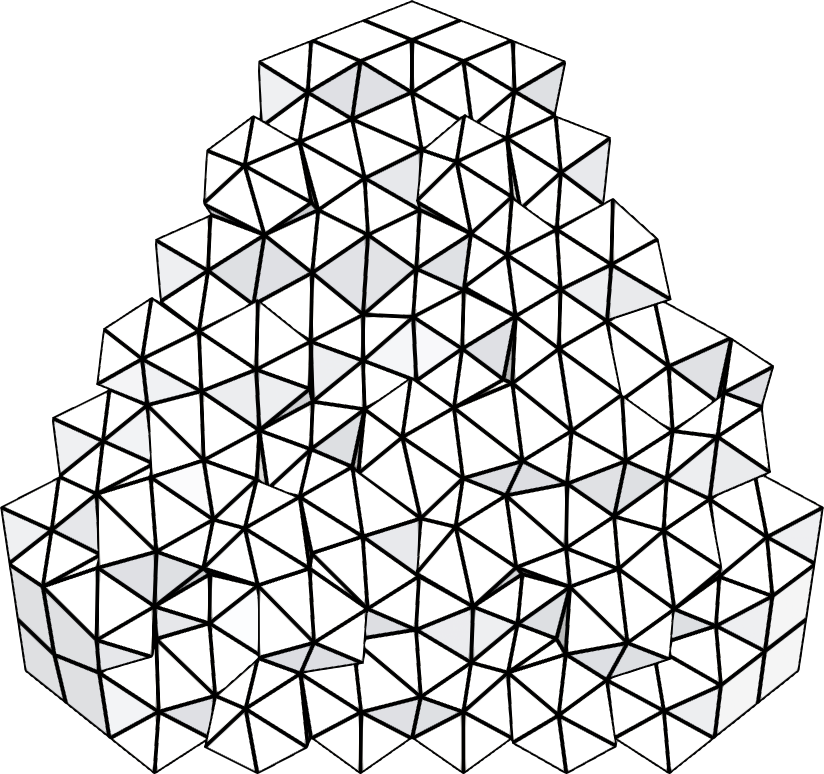}
    \includegraphics[width=0.24\linewidth]{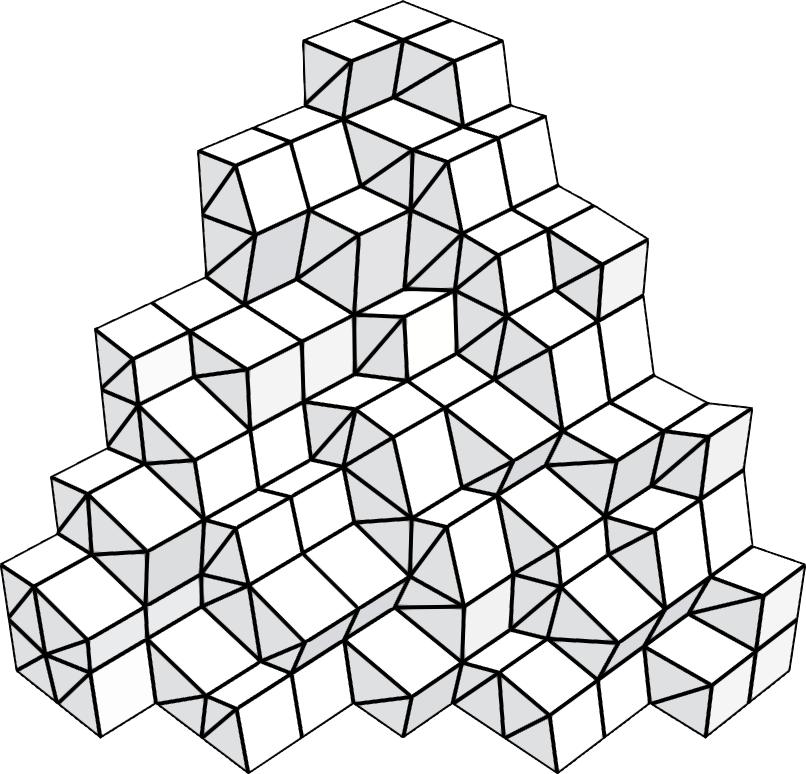}
    \includegraphics[width=0.24\linewidth]{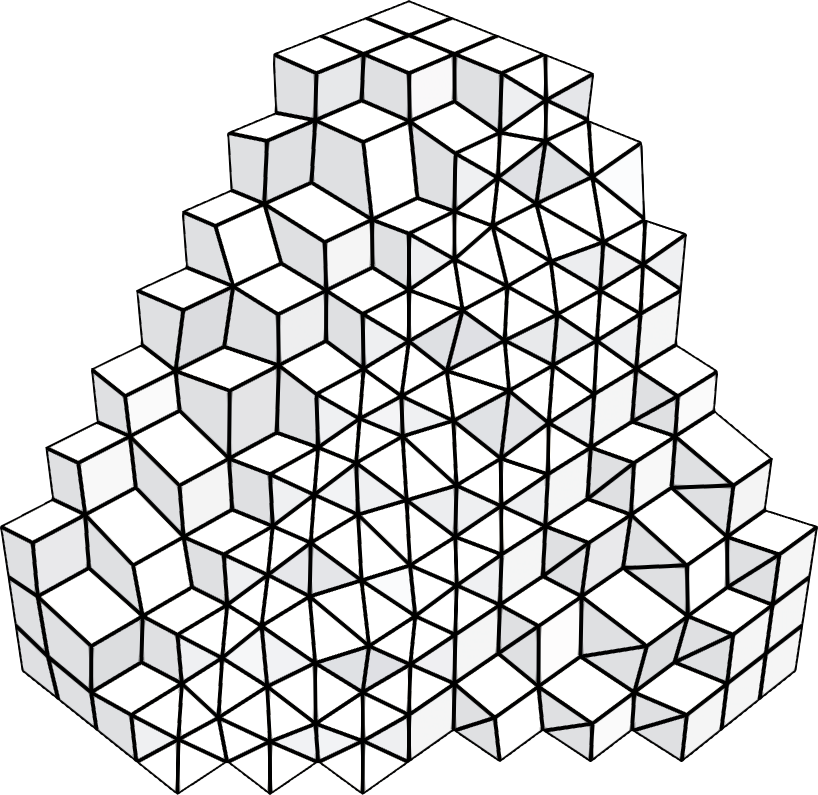} \\
    \includegraphics[width=0.24\linewidth]{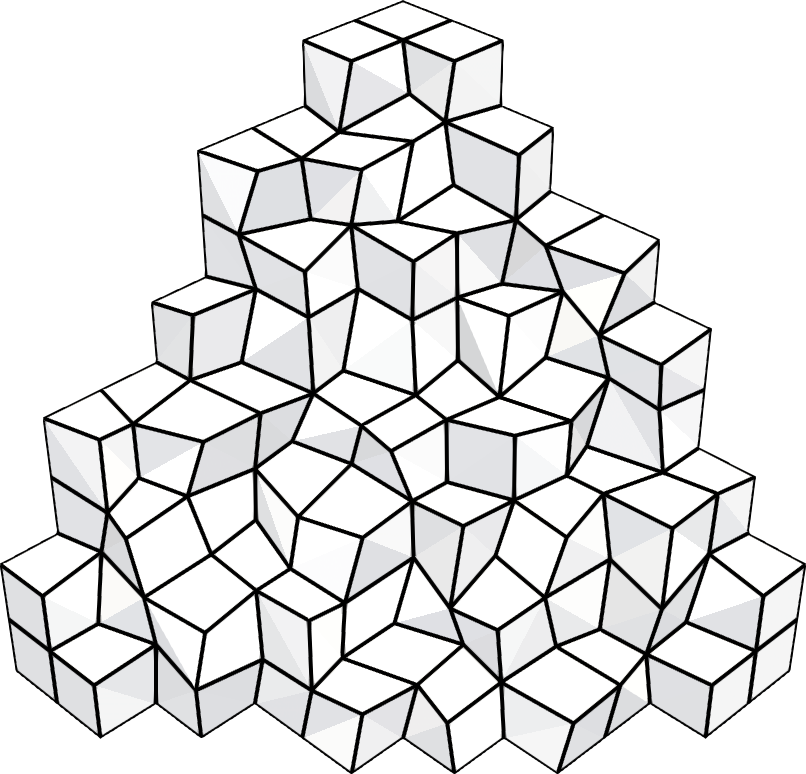}
    \includegraphics[width=0.24\linewidth]{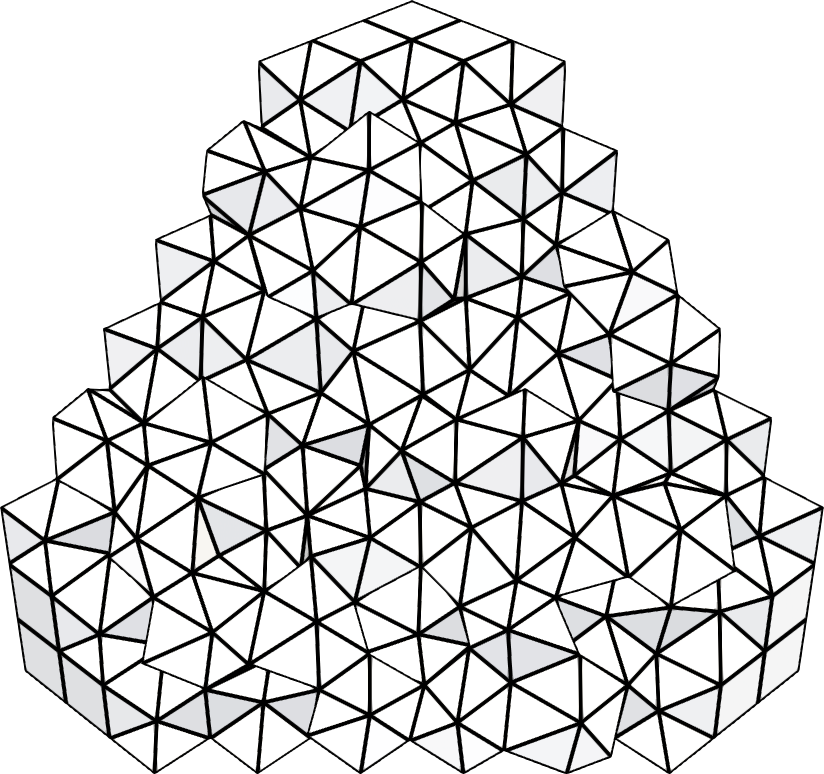}
    \includegraphics[width=0.24\linewidth]{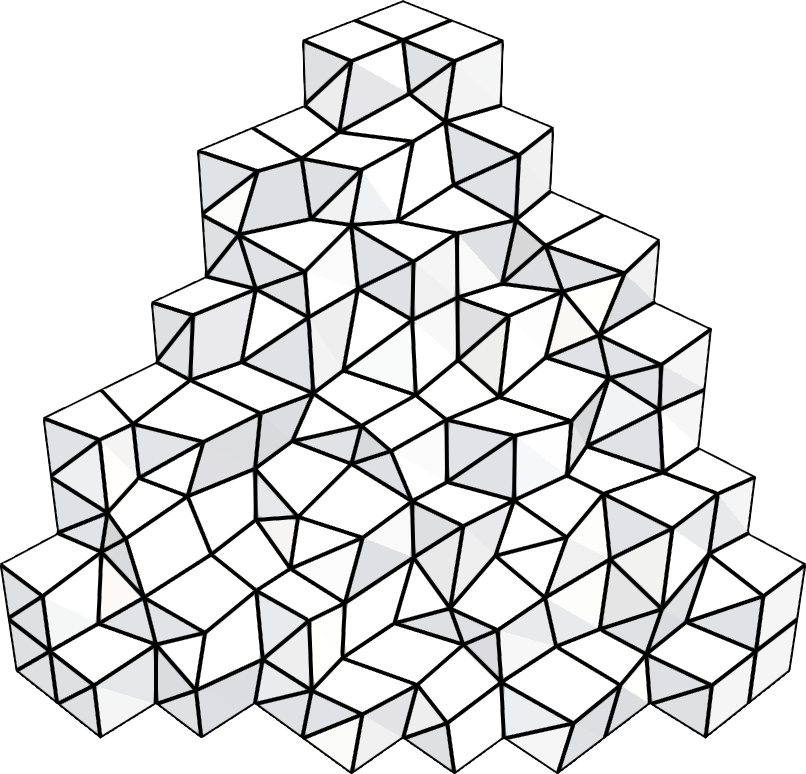}
    \includegraphics[width=0.24\linewidth]{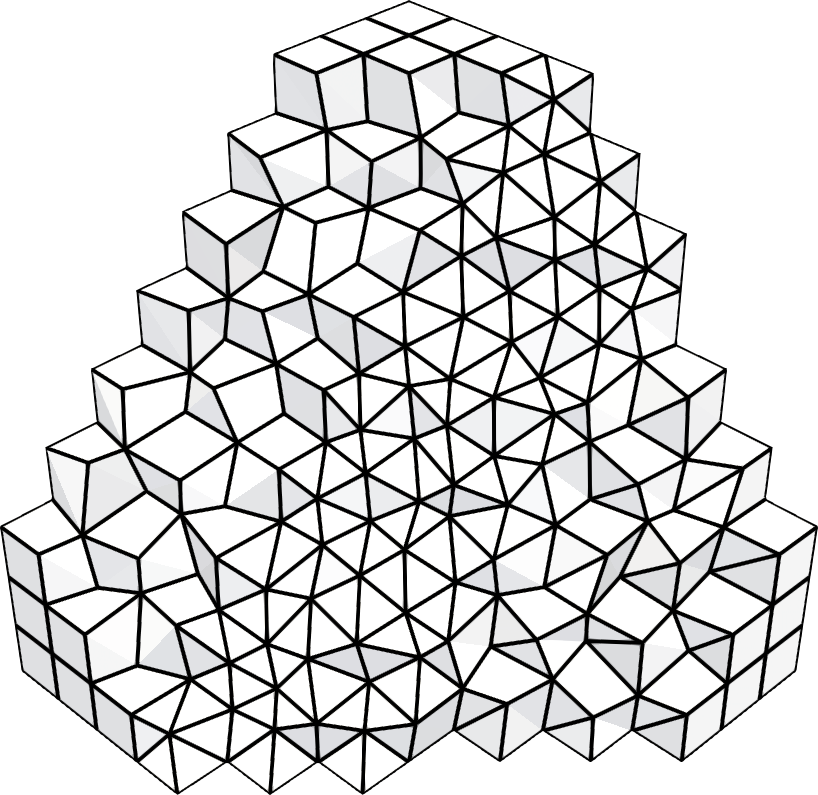} 
    \caption{
    View from above of sections of the coarsest meshes.
    From left to right: hexahedral, pyramidal, prismatic and hybrid meshes. First row: distorted meshes with non-curved faces. Second row: distorted meshes with curved faces.
    }
    \label{fig:meshes}
\end{figure}

The convergence rate in the $H^1(\Omega)$ semi-norm and in the $L^2(\Omega)$ norm are showed in Figures \ref{Poisson:fig:FEM_Poisson_Hexa}, \ref{Poisson:fig:FEM_Poisson_Pri}, \ref{Poisson:fig:FEM_Poisson_Pyr} and \ref{Poisson:fig:FEM_Poisson_Comp} for hexahedral, prismatic, pyramidal and hybrid meshes, respectively.
In all these tests, the optimal orders found from the theory are recovered, both in the $H^1(\Omega)$ semi-norm and in the $L^2(\Omega)$ norm.

\begin{figure}[ht!]
    \centering
    \includegraphics[width = 0.48\textwidth]{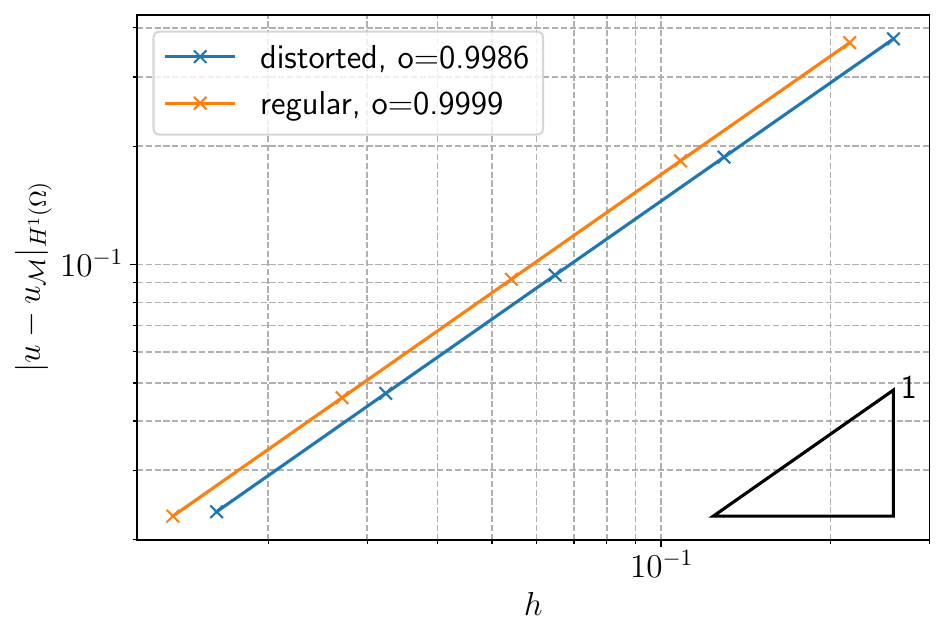}
    \includegraphics[width = 0.48\textwidth]{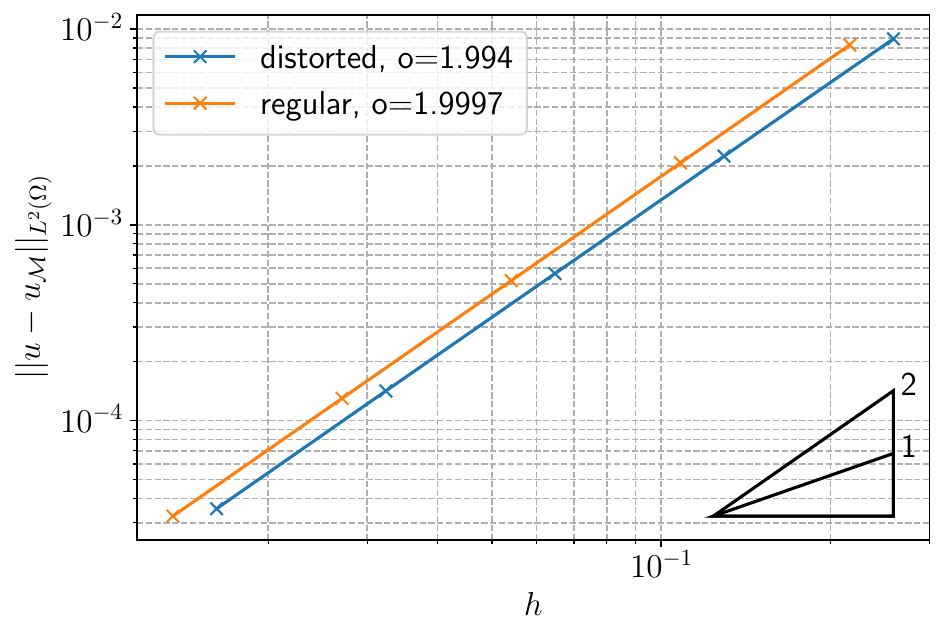}%
    \caption{Convergence rate on a sequence of hexahedral meshes with non-curved faces in the $H^1(\Omega)$ semi-norm and in the $L^2(\Omega)$ norm }
    \label{Poisson:fig:FEM_Poisson_Hexa}
\end{figure}

\begin{figure}[ht!]
    \centering
    \includegraphics[width = 0.48\textwidth]{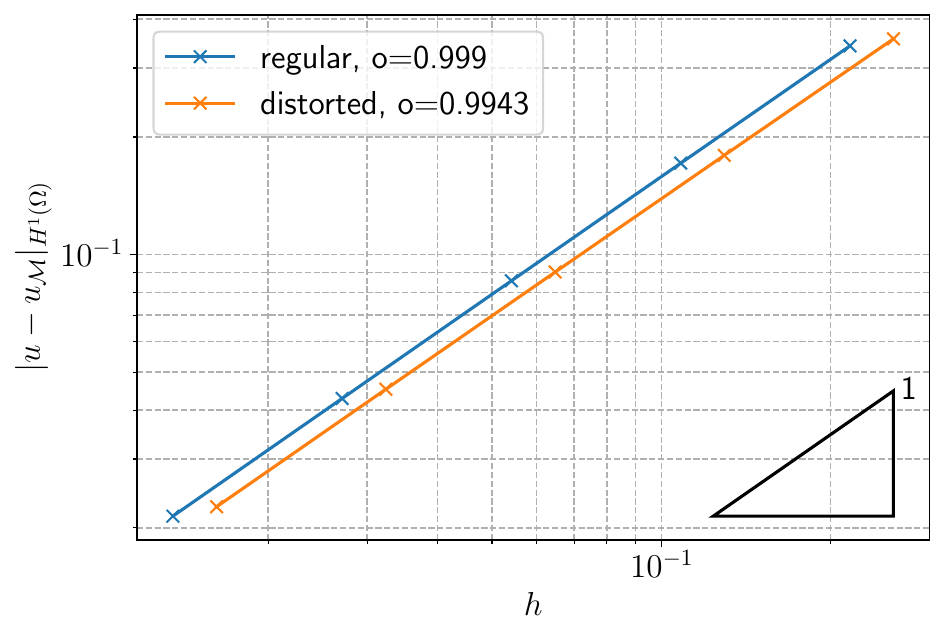}
    \includegraphics[width = 0.48\textwidth]{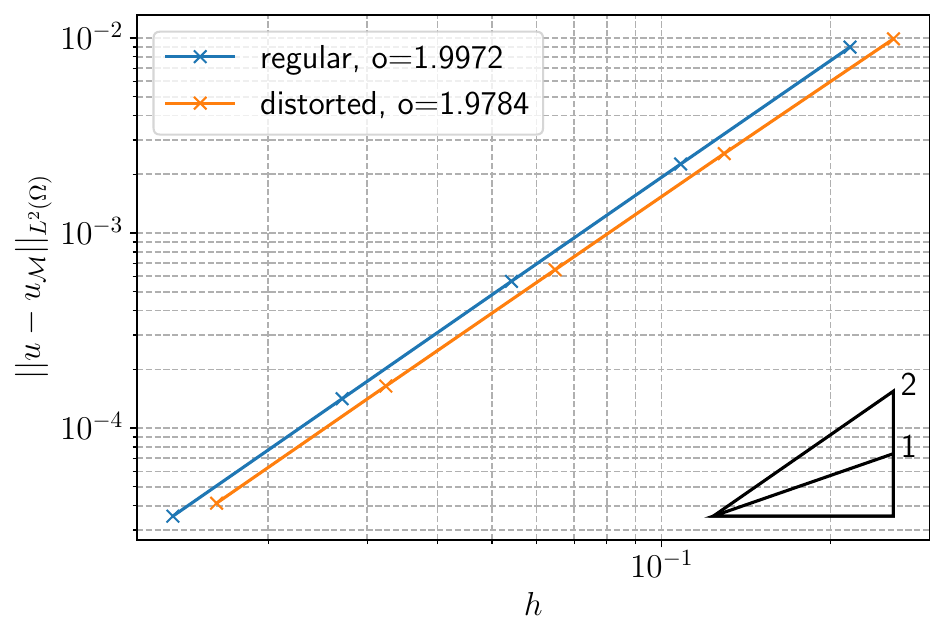}%
    \caption{Convergence rate on a sequence of prismatic meshes with non-curved faces in the $H^1(\Omega)$ semi-norm and in the $L^2(\Omega)$ norm }
    \label{Poisson:fig:FEM_Poisson_Pri}
\end{figure}

\begin{figure}[ht!]
    \centering
    \includegraphics[width = 0.48\textwidth]{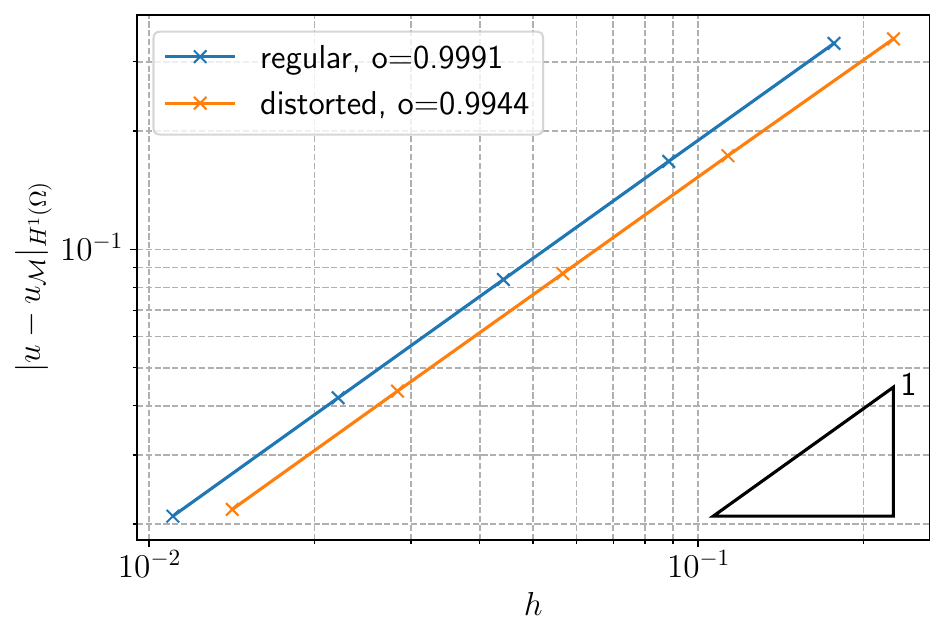}
    \includegraphics[width = 0.48\textwidth]{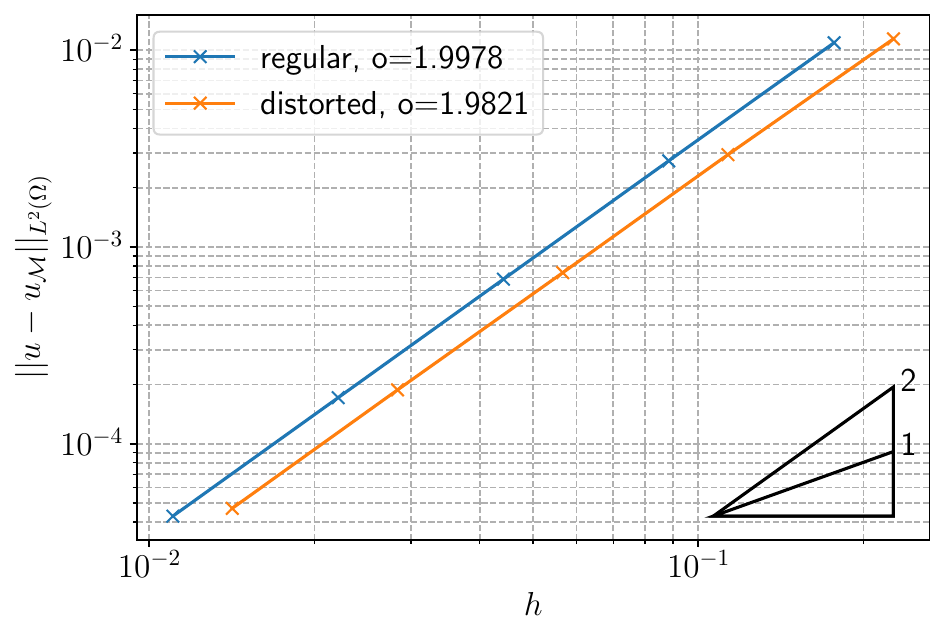}%
    \caption{Convergence rate on a sequence of pyramidal meshes with non-curved faces in the $H^1(\Omega)$ semi-norm and in the $L^2(\Omega)$ norm }
    \label{Poisson:fig:FEM_Poisson_Pyr}
\end{figure}

\begin{figure}[ht!]
    \centering
    \includegraphics[width = 0.48\textwidth]{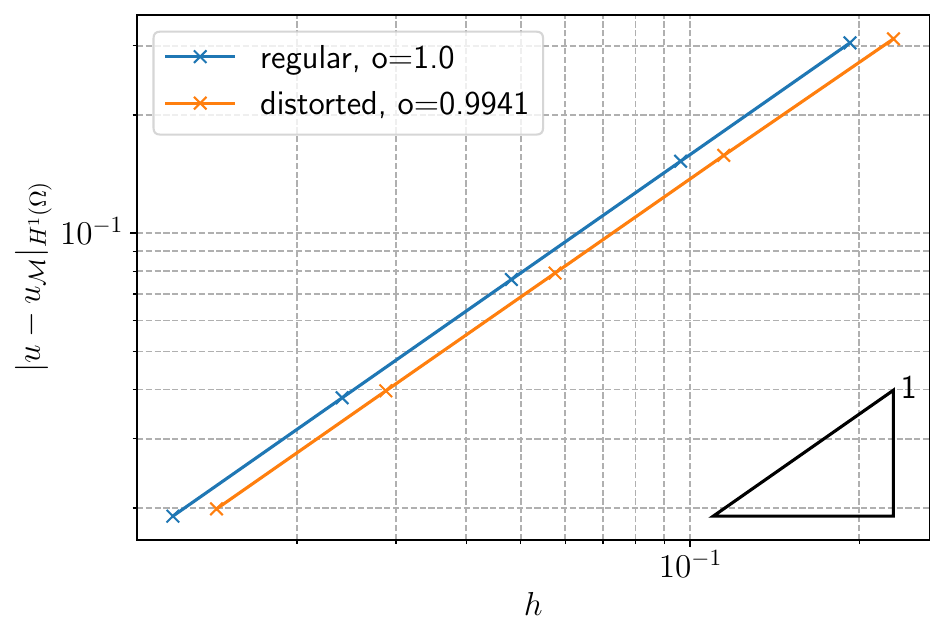}
    \includegraphics[width = 0.48\textwidth]{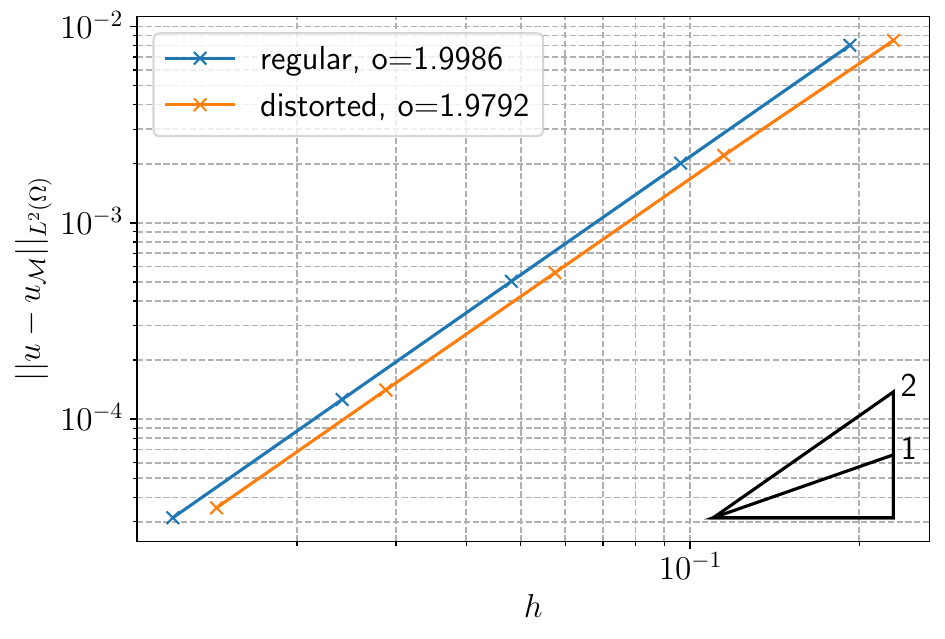}%
    \caption{Convergence rate on a sequence of hybrid meshes with non-curved faces in the $H^1(\Omega)$ semi-norm and in the $L^2(\Omega)$ norm}
    \label{Poisson:fig:FEM_Poisson_Comp}
\end{figure}

\subsection{Test on meshes with curved faces}

We now turn a test on a sequence of meshes with curved faces.
The meshes are constructed in the same way that the sequences of non-distorted meshes from the previous test, except that a distortion (again due to a random displacement of the internal vertex with a length equal to $0.2\,h$) is applied as a final step.
Because the displacement is applied in three spatial dimensions, the vertex of the same face may not belong to the same plane anymore, leading to meshes with curved faces.
Convergence results shown in Figures \ref{Poisson:fig:DG_Poisson_Hexa}, \ref{Poisson:fig:DG_Poisson_Pri}, and \ref{Poisson:fig:DG_Poisson_Pyr} illustrates the lack of consistency of the FE scheme, since the discrete solutions of \eqref{Poisson:eq:disc_Poisson} does not converge.
However, the optimal order of convergence is recovered using the DG-like formulation, where we have set $\gamma = 5$.
We show the results on a sequence of hybrid meshes in Figure \ref{Poisson:fig:DG_Poisson_Composed}.
For this sequence of meshes, the convergence was not reached using $\gamma=5$, so the coefficient was set to $\gamma=10$.
Even with such a choice, the convergence rate in the $H^1(\Omega)$ semi-norm is a little lower than 1 and saturates a little on refined meshes, which suggests that a greater value of $\gamma$ should be used for hybrid meshes.
One may indeed show that a lower bound on the possible values of $\gamma$ to recover stability depends on the mesh regularity parameters, that does not vary in the same way for this sequence of hybrid mesh as for the case of the sequences of non-composite mesh.

\begin{figure}[ht!]
    \centering
    \includegraphics[width = 0.48\textwidth]{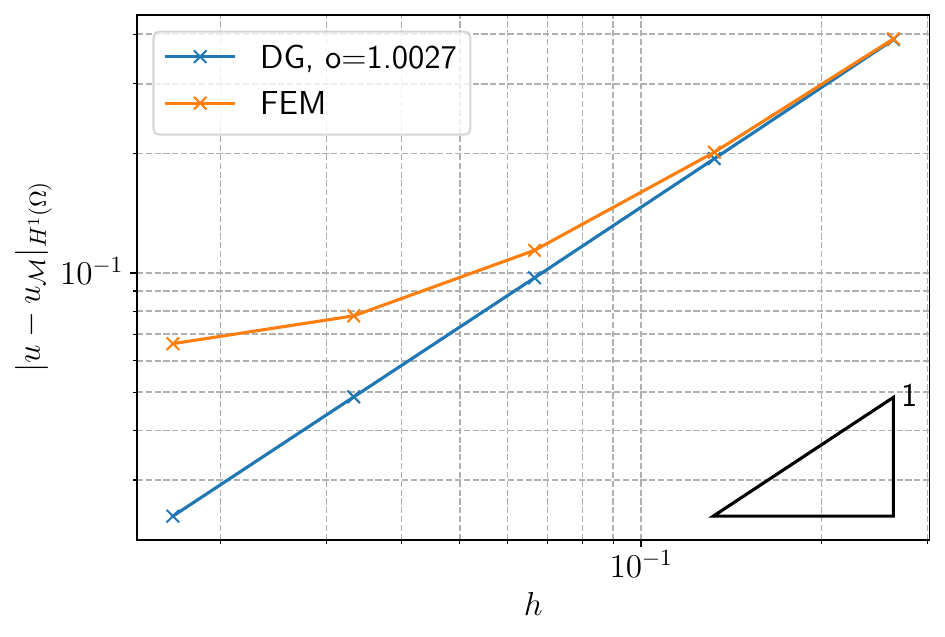}%
    \includegraphics[width = 0.48\textwidth]{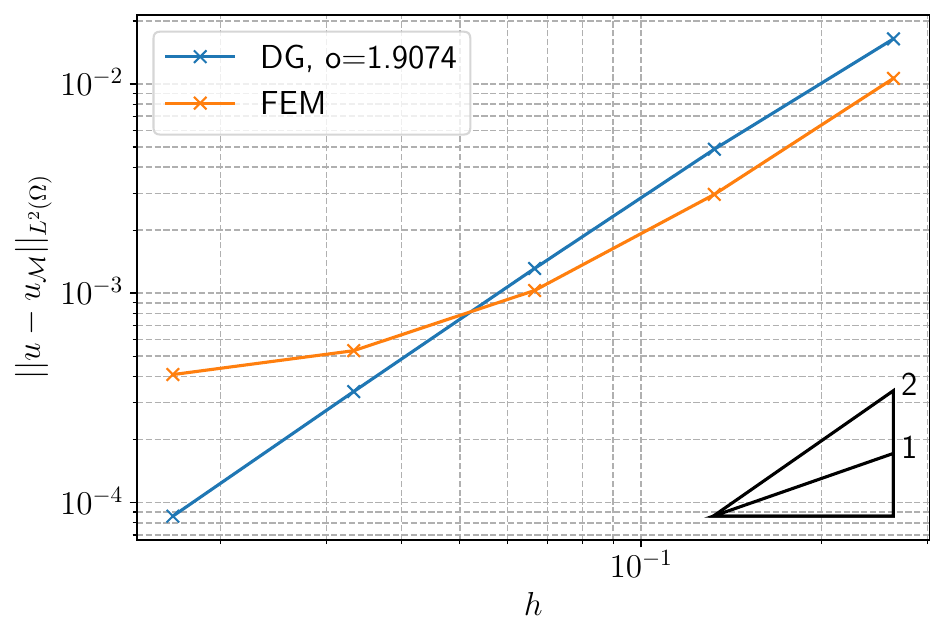}%
    \caption{Convergence rate on a sequence of hexahedral meshes with curved faces in the $H^1(\Omega)$ semi-norm and in the $L^2(\Omega)$ norm }
    \label{Poisson:fig:DG_Poisson_Hexa}
\end{figure}

\begin{figure}[ht!]
    \centering
    \includegraphics[width = 0.48\textwidth]{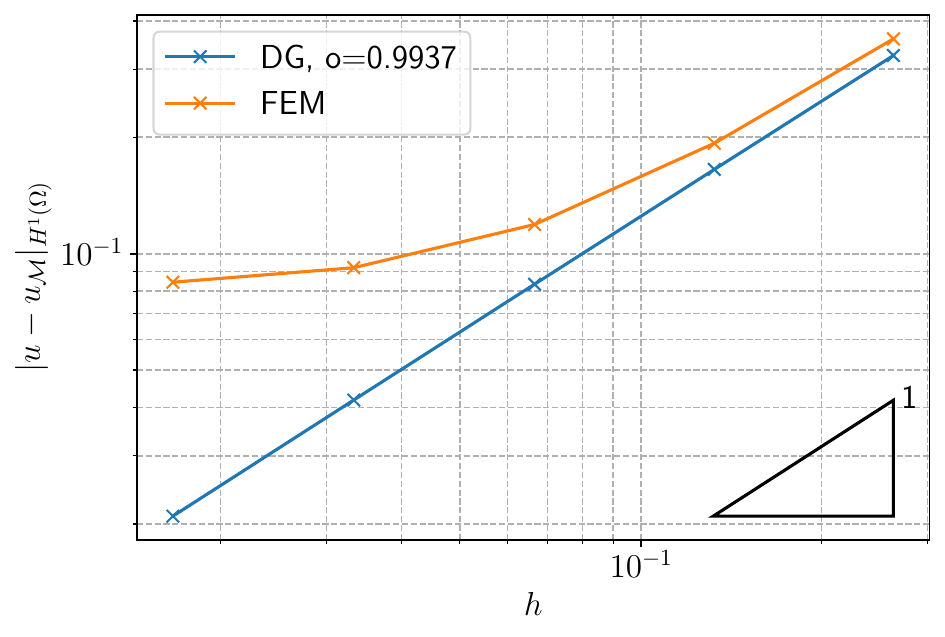}
    \includegraphics[width = 0.48\textwidth]{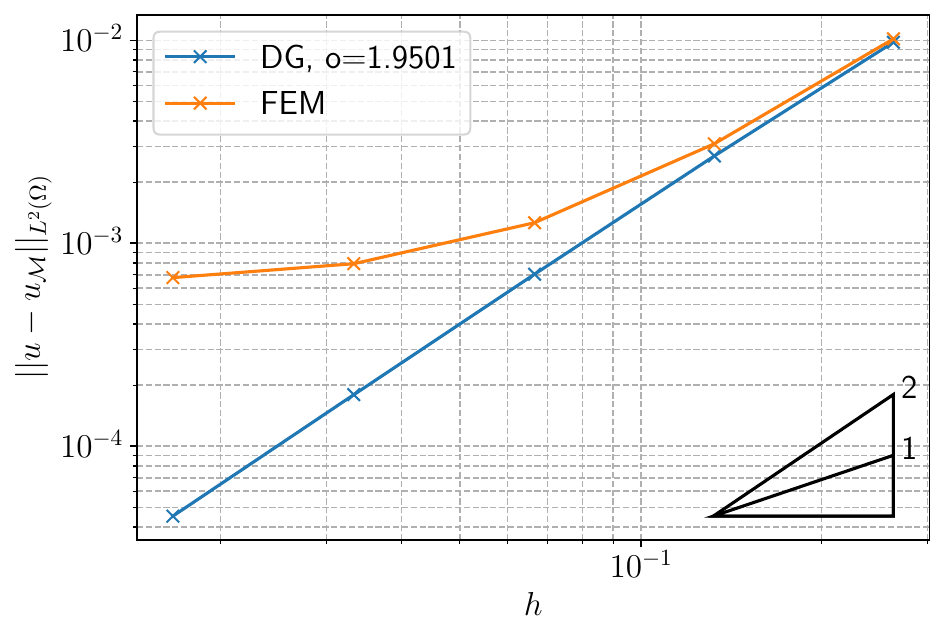}%
    \caption{Convergence rate on a sequence of prismatic meshes with curved faces in the $H^1(\Omega)$ semi-norm and in the $L^2(\Omega)$ norm }
    \label{Poisson:fig:DG_Poisson_Pri}
\end{figure}

\begin{figure}[ht!]
    \centering
    \includegraphics[width = 0.48\textwidth]{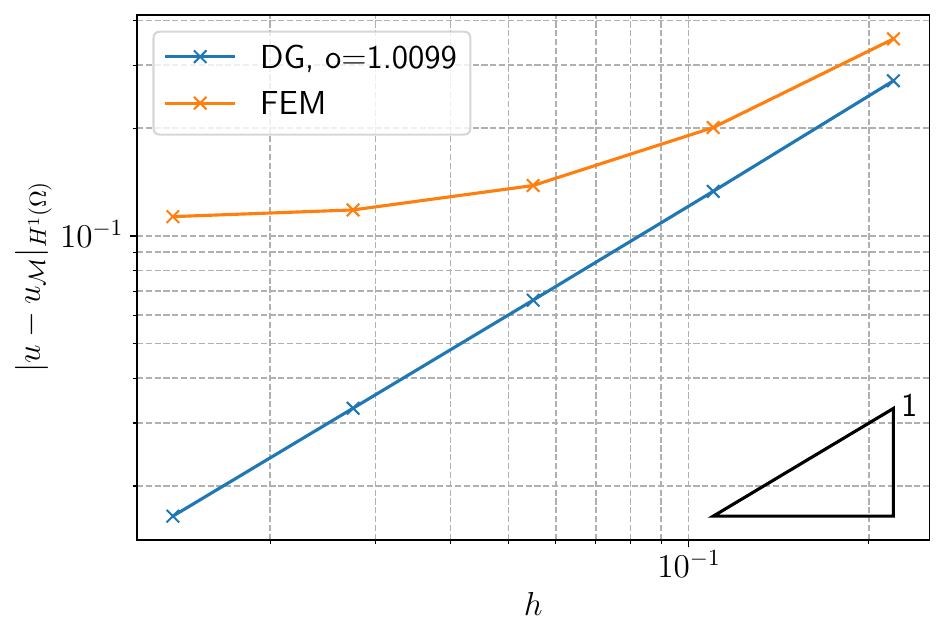}
    \includegraphics[width = 0.48\textwidth]{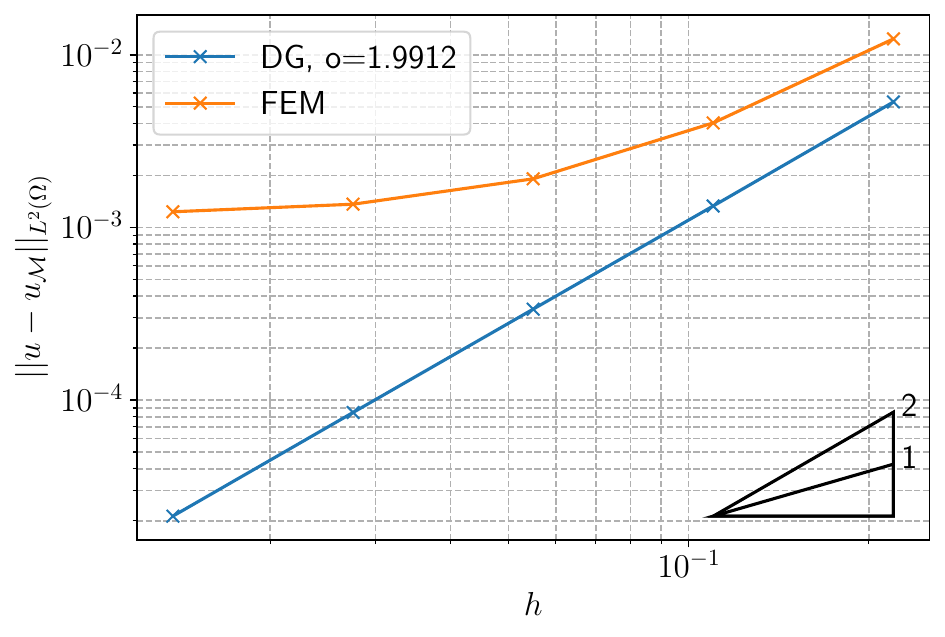}%
    \caption{Convergence rate on a sequence of pyramidal meshes with curved faces in the $H^1(\Omega)$ semi-norm and in the $L^2(\Omega)$ norm }
    \label{Poisson:fig:DG_Poisson_Pyr}
\end{figure}

\begin{figure}[ht!]
    \centering
    \includegraphics[width = 0.48\textwidth]{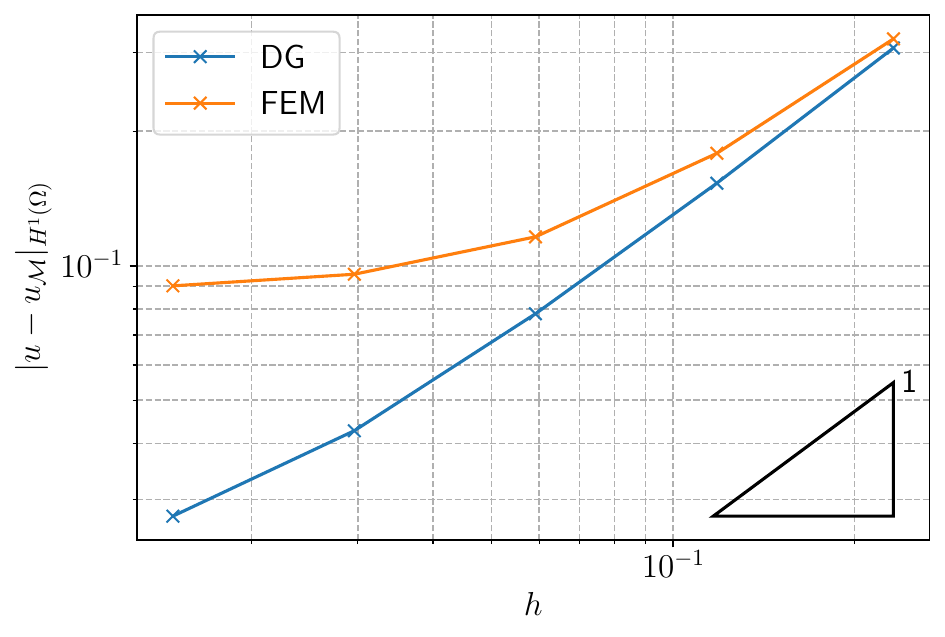}
    \includegraphics[width = 0.48\textwidth]{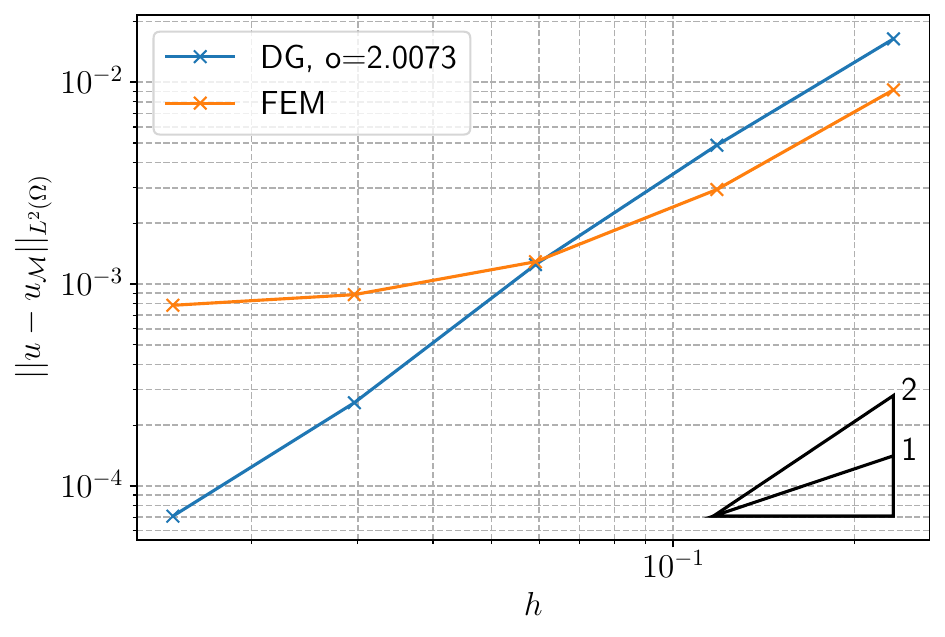}%
    \caption{Convergence rate on a sequence of hybrid meshes with curved faces in the $H^1(\Omega)$ semi-norm and in the $L^2(\Omega)$ norm}
    \label{Poisson:fig:DG_Poisson_Composed}
\end{figure}

\section{Conclusion}

In this work, we have proposed a nonparametric finite element framework designed to handle general meshes.
The construction extends classical Crouzeix-Raviart and Rannacher-Turek elements to arbitrary polytopal cells with a set of face-based degrees of freedom defined solely by mean values on the faces.

We employ a nonparametric approach, which involves constructing local approximation spaces directly on the real cells to incorporate low-order polynomials. 
It allows recovering optimal interpolation error estimates under minimal geometric assumptions. 

Using these results, we have introduced a FE scheme for the Poisson problem on arbitrary polytopal meshes. 
For meshes with planar faces, the method is shown to achieve optimal convergence rates. 
In the presence of curved faces, the scheme becomes non-convergent due to its lack of consistency.
To restore convergence, we proposed introducing DG-like consistency and stabilization terms. 
Numerical experiments confirm the accuracy of both schemes on a variety of three-dimensional meshes.

Then, we gave a method to build the local approximation space in a nonparametric manner, in accordance with the theory.
Explicit examples are given for notably triangular prismatic and quadrilateral pyramidal cells.

The next step of this work is the development of staggered Stokes schemes on general polyhedral meshes, with a view toward applications to time-dependent problems. 
Such schemes can be constructed by defining a velocity approximation space obtained by tensorization of the scalar spaces introduced here, together with a constant-per-cell pressure approximation space.
To achieve this, several key points should be addressed:
\begin{itemize}
    \item The first issue concerns establishing a discrete inf-sup condition for the resulting velocity-pressure pair under minimal mesh regularity assumptions. 
    It can be achieved by showing that the canonical interpolation operator acts as a Fortin operator.
    The proof can be derived in a straightforward manner thanks to the inclusion $\P_1(K) \subset P_K$.
    
    \item  The second issue concerns the error estimates of the resulting finite element scheme. 
    The interpolation error and the velocity consistency error can be estimated using arguments analogous to those developed in the present work. 
    In addition, suitable bounds must be derived for the pressure-velocity coupling bilinear form. 
    This analysis can be carried out in a straightforward manner when the polytopal cells have non-curved faces.
    
    \item However, the situation is not as clear in the presence of curved faces. 
    Indeed, the pressure-velocity coupling term can be rewritten, in a finite-volume fashion, as a discrete velocity divergence operator or as a discrete pressure gradient operator. 
    It is desirable to define the discrete velocity divergence operator so that it yields locally conservative fluxes across mesh faces and preserves an $L^2$-duality with the discrete pressure gradient operator. 
    These properties are essential, for instance, to derive kinetic energy estimates for time-dependent problems. 
    A naive recovery of consistency through DG-like stabilization terms may break one or both of these properties, or even violate the inf-sup condition. 
    Special care is therefore required to design the velocity-pressure coupling bilinear form.

    \item Finally, for certain applications, the CR and the RT elements are known to suffer from spurious pressure oscillations, sometimes referred to as checkerboard modes. 
    These artifacts arise from an incompatibility between the velocity and pressure approximation spaces, typically due to an imbalance in their respective numbers of degrees of freedom. 
    Similar issues may occur when using prismatic or pyramidal cells; strategies to mitigate such effects should thus maybe be derived. 
\end{itemize}

\begin{acknowledgement}
\textbf{Acknowledgement} -- 
This work was conducted during my (the first author) final year of PhD under the guidance of Jean-Claude, who passed away in October 2025. I am deeply grateful for his mentorship and his support throughout these years. Proofreading and shaping this article without his expertise proved challenging, but I hope this work meets the standards of excellence that Jean-Claude set throughout his career.    
\end{acknowledgement}

\appendix

\section{Building three-dimensional cells}
\label{app:mapping}

In this appendix, following the framework of \cite{ern-21-fin}, we recall the construction of three-dimensional meshes.
In practice, the assembly of finite element operators or the computation of a cell volume rely on the evaluation of integrals over physical cells, which requires (i) quadrature rules on a reference configuration and (ii) a mapping from this reference configuration to the cell under consideration.
The construction of quadrature formulae is not addressed here; for classical triangular prismatic and quadrilateral pyramidal cells, suitable rules can be found in \cite{kub-13-new,kub-13-new2,fel-04-com,jas-21-hig}.
Rather, we focus on the definition of the mappings, defined solely by the vertices of the cell under consideration. 
Indeed, mesh-generation tools usually only provide the position of the vertices of the cells.
While these unambiguously define a cell when all the vertices of a face of that cell are co-planar, it is not the case when non-co-planar vertices are involved.

For tetrahedra and hexahedra, standard choices include the $\P_1$ mapping and the $\Q_1$ mapping \cite{cia-91-han}, respectively; mappings for prisms and pyramids are, on the other hand, not as widely common.
Beyond being defined entirely by the vertices of the real cells, these mappings should also satisfy two different properties.
First, on so-called regular cells, the mapping should be affine, as its Jacobian is constant in such a case, which, in turn, allows the efficient use of low-order quadrature methods.
Second, the mappings must be conforming: for two adjacent cells, it means that their common face must be uniquely and consistently defined, independently of the cells' natures.
This latest requirement is formalized by the following criterions.
\begin{crtrn}[Conformity with respect to triangular faces] \label{hyp:cons_tri}
    Let $\edge \in \edges$ be a triangular face with vertices $\bfv_0$, $\bfv_1$, and $\bfv_2$. 
    Regardless of it(s) neighboring cell(s), and up to a permutation of its vertices, $\edge$ can be defined as the affine triangle:
    \begin{align*}
        \edge = \{ \bfv_0 + (\bfv_1 - \bfv_0) \widehat x + (\bfv_2 - \bfv_0) \widehat y, \; 0 \leq \widehat x,\widehat y \leq 1 \}.
    \end{align*}
\end{crtrn}
\begin{crtrn}[Conformity with respect to quadrangular faces] \label{hyp:cons_quad}
    Let $\edge \in \edges$ be a quadrilateral face with vertices $\bfv_0$, $\bfv_1$, $\bfv_2$, and $\bfv_3$. 
    Regardless of it(s) neighboring cell(s), and up to a permutation of its vertices, $\edge$ can be defined as the bilinear surface (here-called "generalized quadrilateral"):
    \begin{align*}
        \edge = \{ \bfv_0 + (\bfv_1 - \bfv_0) \widehat x + (\bfv_3 - \bfv_0) \widehat y + ( \bfv_2 - \bfv_1 + \bfv_0 - \bfv_3) \widehat x \widehat y, \; 0 \leq \widehat x, \widehat y \leq 1 \}.
    \end{align*}
\end{crtrn}

Let us now detail the construction of such mappings.
Let $\mesh$ be a three-dimensional mesh and let $K \in \mesh$.
Denote by $\widehat K^{geo}$ a geometric reference cell of the same polyhedral nature as $K$.
Our goal is to define $T_K: \widehat K^{geo} \to K$ entirely determined by the vertices of $K$.
To do so, we introduce a geometric element $(\widehat K, P_{\widehat K}^{geo}, \Sigma_{\widehat K}^{geo})$, which should not be confused with the finite element used for the local approximation spaces. 
Here,
\begin{itemize}
    \item $\widehat K^{geo}$ is a simple reference cell;
    \item $P_{\widehat K}^{geo}$ is a space of functions as simple as possible (\ie, when possible, polynomials of low order);
    \item $\Sigma_{\widehat K}^{geo}$ is the set of degrees of freedom associated with the vertices of the reference cell.
\end{itemize}
Assume now that such an element is available, that is to say we have built a set of functions $(\widehat \Psi_i)_{0 \leq i \leq n_V -1}$, with $n_V$ the number of vertices of the considered polyhedra, such that we have $\widehat \Psi_i(\widehat v_j) = \delta_{i,j}$ for all vertices $(\widehat v_j)_{0 \leq j \leq n_V -1}$ of the reference cell $K$.
Then, for a cell $K$ and its vertices $({\bfv_i})_{0 \leq i \leq n_V-1}$, the mapping may be defined as:
\begin{align} \label{Poisson:eq:transfo}
    T_K(\widehat \bfx) = \sum_{i=0}^{n_V-1} \widehat \Psi_i(\widehat \bfx) \bfv_i,
\end{align}
where the order of the vertices are ordered similarly for the real-world cell $K$ and for the reference cell $K$ (that is, we intend to have $\bfv_i =T_K(\widehat \bfv_i)$ for all $0 \leq i \leq n_V-1$).

In the following, we specify, for each type of the considered three-dimensional cells, the space $P_{\widehat K}^{geo}$, the vertices $( \widehat \bfv_i)_{0 \leq i \leq n_V-1}$ of the reference cell, and the associated shape functions $(\widehat \Psi_i)_{0 \leq i \leq n_V-1}$.
We also show that the resulting mappings satisfy the above criterions.

\subsection{Mapping for tetrahedal cells}

For tetrahedra, we adopt the classical $\P_1$ construction and set $P_{\widehat K}^{geo} = \text{span}\left\{1, \widehat x, \widehat y, \widehat z \right\}$.
The vertices and the shape functions are listed in Table \ref{tab:transfo_tetra}.
Using these functions in the definition of the mapping \eqref{Poisson:eq:transfo}, then substituting in order $\widehat x = 0$, $\widehat y = 0$,  $\widehat z = 0$ and $\widehat z = 1 - \widehat x - \widehat y$ shows that the mapping satisfies Criterion \ref{hyp:cons_tri}.
By definition, this mapping is always affine.

\begin{table}[ht!]
    \centering
    \begin{tabular}{c|l}
        $\widehat{\bfv_0}=(0,0,0)$ & $\Psi_{0}(\widehat x, \widehat y, \widehat z) = 1 - \widehat x - \widehat y - \widehat z $ \\
        $\widehat{\bfv_1}=(1,0,0)$ & $\Psi_{1}(\widehat x, \widehat y, \widehat z) = \widehat x $ \\
        $\widehat{\bfv_2}=(0,1,0)$ & $\Psi_{2}(\widehat x, \widehat y, \widehat z) = \widehat y$ \\
        $\widehat{\bfv_3}=(0,0,1)$ & $\Psi_{3}(\widehat x, \widehat y, \widehat z) = \widehat z$ \\
    \end{tabular}
    \caption{Vertices and shape functions for the reference tetrahedra}
    \label{tab:transfo_tetra}
\end{table}

\subsection{Mapping for hexahedral cells}

For hexahedra, we employ $\Q_1$ polynomials, so that $P_{\widehat K}^{geo} = \text{span}\left\{1, \widehat x, \widehat y, \widehat z, \widehat x \widehat y, \widehat x \widehat z, \widehat y \widehat z \right\}$.
The vertices and the shape functions are given in Table \ref{tab:transfo_hexa}.
A direct computation of the mapping \eqref{Poisson:eq:transfo} with these functions, substituting in order $\widehat x = 0$, $\widehat x = 1$, $\widehat y = 0$, $\widehat y = 1$, $\widehat z = 0$ and $\widehat z = 1$, shows that the mapping satisfies Criterion \ref{hyp:cons_quad}.
The mapping reduces to an affine transformation whenever the real cell is a parallelepiped.

\begin{table}[ht!]
    \centering
    \begin{tabular}{c|l}
        $\widehat{\bfv_0}=(0,0,0)$ & $\Psi_{0}(\widehat x, \widehat y, \widehat z) = (1 - \widehat x)(1 - \widehat y)(1 - \widehat z)$ \\
        $\widehat{\bfv_1}=(0,1,0)$ & $\Psi_{1}(\widehat x, \widehat y, \widehat z) = (1 - \widehat x)\widehat y (1 - \widehat z)$ \\
        $\widehat{\bfv_2}=(0,0,1)$ & $\Psi_{2}(\widehat x, \widehat y, \widehat z) = (1 - \widehat x)(1 - \widehat y)\widehat z$ \\
        $\widehat{\bfv_3}=(1,0,0)$ & $\Psi_{3}(\widehat x, \widehat y, \widehat z) = \widehat x(1 - \widehat y) (1 - \widehat z)$ \\
        $\widehat{\bfv_4}=(1,1,0)$ & $\Psi_{4}(\widehat x, \widehat y, \widehat z) = \widehat x \widehat y (1 - \widehat z)$ \\
        $\widehat{\bfv_5}=(1,0,1)$ & $\Psi_{5}(\widehat x, \widehat y, \widehat z) = \widehat x (1- \widehat y) \widehat z$ \\
        $\widehat{\bfv_6}=(0,1,1)$ & $\Psi_{6}(\widehat x, \widehat y, \widehat z) = (1 - \widehat x)\widehat y \widehat z$ \\
        $\widehat{\bfv_7}=(1,1,1)$ & $\Psi_{7}(\widehat x, \widehat y, \widehat z) = \widehat x \widehat y \widehat z$ \\
    \end{tabular}
    \caption{Vertices and shape functions for the reference cube}
    \label{tab:transfo_hexa}
\end{table}

\subsection{Mapping for quadrilateral pyramid cells}

For quadrilateral pyramidal cells, we adopt the space of Bedrosian functions (see \cite{bed-92-sha}), although alternative constructions exist (see for instance \cite{kna-01-inv}).
This space is not polynomial, and is defined as $P_{\widehat K}^{geo} = \text{span}\left\{1, \widehat x, \widehat y, \widehat z, \frac{\widehat x \widehat y}{1 - \widehat z}\right\}$.
The vertices and the shape functions are reported in Table \ref{tab:transfo_pyr}.
As explained for instance in \cite{ber-10-hig}, the shape functions may be derived from the $\Q_1$ shape functions for the reference cube through a non-bijective mapping from the reference cube to the reference pyramid defined as $T_{pyr}(\widehat x, \widehat y, \widehat z) = ((1 - \widehat z) \widehat y, (1 - \widehat z) \widehat y, \widehat z)^t$. 
Using these basis functions in the mapping \eqref{Poisson:eq:transfo}, then substituting in order $\widehat x = 0$, $\widehat y = 0$, $ \widehat z = 1 - \widehat x$, $\widehat z = 1 - \widehat y$, one verifies that the mapping satisfies Criterion \ref{hyp:cons_tri}.
Substituting $\widehat z = 0$ shows that it satisfies as well Criterion \ref{hyp:cons_quad}.
Elementary computation (see \cite{ber-10-hig}) shows that the mapping is affine whenever the quadrilateral face of the pyramid is a parallelogram.

\begin{table}[ht!]
    \centering
    \begin{tabular}{c|l}
        $\widehat{\bfv_0}=(0,0,0)$ & $\Psi_0( \widehat x, \widehat y, \widehat z) = \left(1 - \dfrac{\widehat x}{1 - \widehat z}\right)\left(1 - \dfrac{\widehat y}{1 - \widehat z}\right)(1- \widehat z)$ \\
        $\widehat{\bfv_1}=(1,0,0)$ &  $\Psi_1( \widehat x, \widehat y, \widehat z) = \dfrac{\widehat x}{1 - \widehat z}\left(1 - \dfrac{\widehat y}{1 - \widehat z}\right)(1- \widehat z)$ \\
        $\widehat{\bfv_2}=(1,1,0)$ & $\Psi_2( \widehat x, \widehat y, \widehat z) = \dfrac{\widehat x}{1 - \widehat z} \dfrac{\widehat y}{1 - \widehat z} (1- \widehat z)$ \\
        $\widehat{\bfv_3}=(0,1,0)$ & $\Psi_3( \widehat x, \widehat y, \widehat z) = \left(1 - \dfrac{\widehat x}{1 - \widehat z}\right)\dfrac{\widehat y}{1 - \widehat z}(1-\widehat z)$ \\
        $\widehat{\bfv_4}=(0,0,1)$ & $\Psi_4( \widehat x, \widehat y, \widehat z) = \widehat z$ \\
    \end{tabular}
    \caption{Vertices and shape functions for the reference pyramid}
    \label{tab:transfo_pyr}
\end{table}

\subsection{Mapping for prismatic cells}

For prismatic cells, we define $P_{\widehat K}^{geo} = \text{span}\left\{1 , \widehat x, \widehat y, \widehat z, \widehat x \widehat y, \widehat x \widehat z \right\}$.
The vertices and the polynomials $(\Psi_i)_{1 \leq i \leq 6}$ are given in Table \ref{tab:transfo_pri}.
The polynomials may be directly computed once again thanks to the $\Q_1$ mapping for the reference cube through a non-bijective mapping from the reference cube to the reference prism defined as $T_{pri}(\widehat x, \widehat y, \widehat z) = (\widehat x, (1 - \widehat z) \widehat y, \widehat z)^t$.
Taking these functions in the definition of the mapping \eqref{Poisson:eq:transfo}, then substituting in order $\widehat x = 0$, and $\widehat x = 1$ shows that the mapping satisfies Criterion \ref{hyp:cons_tri}.
Substituting $\widehat y =0$, $\widehat z = 0$ and then $\widehat y = 1 - \widehat z$ shows that it satisfies as well Criterion \ref{hyp:cons_quad}.
This mapping is affine whenever the two triangular faces lies in parallel planes.

\begin{table}[ht!]
    \centering
    \begin{tabular}{c|l}
        $\widehat{\bfv_0}=(0,0,0)$ & $\Psi_0(\widehat x, \widehat y, \widehat z) = (1 - \widehat x)(1 - \widehat y - \widehat z)$ \\
        $\widehat{\bfv_1}=(0,1,0)$ & $\Psi_1(\widehat x, \widehat y, \widehat z) = (1 - \widehat x)\widehat y$ \\
        $\widehat{\bfv_2}=(0,0,1)$ & $\Psi_2(\widehat x, \widehat y, \widehat z) = (1 - \widehat x)\widehat z$ \\
        $\widehat{\bfv_3}=(1,0,0)$ & $\Psi_3(\widehat x, \widehat y, \widehat z) = \widehat x(1 - \widehat y - \widehat z)$ \\
        $\widehat{\bfv_4}=(1,1,0)$ & $\Psi_4(\widehat x, \widehat y, \widehat z) = \widehat x \widehat y$ \\
        $\widehat{\bfv_5}=(1,0,1)$ & $\Psi_5(\widehat x, \widehat y, \widehat z) = \widehat x \widehat z$ \\
    \end{tabular}
    \caption{Vertices and shape functions for the reference prism}
    \label{tab:transfo_pri}
\end{table}

\section{A trace lemma on possibly curved faces}
\label{app:trace}
In this Appendix, we adapt the proof of the trace lemma given in \cite{ey-09-con} to curved faces. 
Contrary to \cite{can-21-hp}, we do not suppose here that the cells are star shaped.
However, we introduce the slightly less restrictive notion of truncated cones.
The result is as follows:

\begin{lmm}[Trace lemma on curved faces] \label{lem:trace_curved}
Let $K \in \mesh$.
Let $\edge \in \edges(K)$ be a face such that there exists $T \in \mathcal{C}^1([0,1]^2,\R^3)$ such that $\edge := T(\widehat \edge )$ with $\widehat \edge := \{\miny \leq y \leq \maxy, \; \minx \leq x \leq \maxx \} \subset [0,1]^2$ with $\miny, \maxy \in [0,1]^2$ and $\underline m_{\widehat x},\overline m_{\widehat x} \in \left(\mathcal{C}^0([\miny,\maxy],[0,1])\right)^2$.
Let $\alpha \in ]0,1[$ and $\bfx_{\edge,c_\alpha} \in \Omega \backslash \edge$ be such that the truncated cone
\begin{align*}
    C(\edge,c_\alpha):=\left\{ t \bfx + (1-t) \bfx_{\edge,c_\alpha}, \; \bfx \in \edge, \; t \in [\alpha,1] \right\},
\end{align*}
is contained in $K$.
Then, for all $f \in H^1(K)$ and $\varepsilon > 0$, the following inequality holds:
\begin{align*}
    \norm{f}_{L^2(\edge)}^2 \leq \frac{1}{\alpha^2 a(\edge,\bfx_{\edge,c_\alpha}) d(\edge,\bfx_{\edge,c_\alpha})} \left( \frac{1 + \varepsilon }{1-\alpha}\norm{f}_{L^2(C(\edge,c_\alpha))}^2 + \frac{D(\edge,\bfx_{\edge,c_\alpha})^2}{\varepsilon}   \norm{\nabla f}_{L^2(C(\edge,c_\alpha))}^2 \right),
\end{align*}
where 
\begin{align*}
    d(\edge,\bfx_{\edge,c_\alpha}) & = \inf_{\bfx \in \edge}|\bfx - \bfx_{\edge,c_\alpha}|, \\
    D(\edge,\bfx_{\edge,c_\alpha}) & = \sup_{\bfx \in \edge}|\bfx - \bfx_{\edge,c_\alpha}|, \\
    a(\edge,\bfx_{\edge,c_\alpha}) & = \inf_{\bfx \in \edge}\left| \frac{\bfx_{\edge,c_\alpha} - \bfx}{|\bfx_{\edge,c_\alpha} - \bfx|} \cdot \bfn_\edge(\bfx) \right|.
\end{align*}
\end{lmm}
\begin{proof}
We prove the lemma using a density argument, so that we consider in the following that $f \in C^1(K)$.
Let $(\widehat x,\widehat y) \in \widehat \edge$ and $g: [\alpha, 1] \to \R$ be the function defined by $g(t) = \frac{t - \alpha}{1- \alpha} |f(\gamma(t))|^2$ where $\gamma(t) = t T(\widehat x,\widehat y) + (1-t) \bfx_{\edge,c_\alpha}$.
We then have:
\begin{align*}
    |f(T(\widehat x,\widehat y))|^2 
    & = \int_{\alpha}^1 g'(t) \dd t \\
    & = \int_{\alpha}^1 \frac{1}{1-\alpha} |f(\gamma(t))|^2 \dd t + \int_{\alpha}^1 2 \frac{t-\alpha}{1-\alpha} \nabla_{\gamma(t)}f \cdot (T(\widehat x,\widehat y)-\bfx_{\edge,c_\alpha}) f(\gamma(t)) \dd t.
\end{align*}
Let us recall that for a function $h \in L^1(\edge)$, its integral on $\edge$ is defined as:
\begin{align*}
    \int_\edge h(\bfx) \dd \gamma(\bfx) = \int_{\miny}^{\maxy} 
    \int_{\minx}^{\maxx} h(T(\widehat x, \widehat y)) \left| \frac{d T(\widehat x, \widehat y)}{d \widehat x}  \wedge \frac{d T(\widehat x, \widehat y)}{d \widehat y} \right| \dd \widehat x \dd \widehat y.
\end{align*}
Taking $h = |f|^2$, it follows that
\begin{align*}
    \int_\edge |f(\bfx)|^2 \dd \gamma(\bfx) = I_1 + I_2,
\end{align*}
with
\begin{align*}
    I_1 & =  \frac{1}{1-\alpha} \int_{\miny}^{\maxy} \int_{\minx}^{\maxx} \int_{\alpha}^1 |f(\gamma(t))|^2  \left| \frac{d T(\widehat x,\widehat y)}{d \widehat x}  \wedge \frac{d T(\widehat x,\widehat y)}{d \widehat y} \right| \dd t \dd \widehat x \dd \widehat y , \\
    I_2 & = 2 \int_{\miny}^{\maxy} \int_{\minx}^{\maxx} \int_{\alpha}^1 \frac{t-\alpha}{1-\alpha} \nabla_{\gamma(t)}f \cdot (T(\widehat x,\widehat y)-\bfx_{\edge,c_\alpha}) f(\gamma(t))  \left| \frac{d T(\widehat x,\widehat y)}{d \widehat x}  \wedge \frac{d T(\widehat x,\widehat y)}{d \widehat y} \right| \dd t \dd \widehat x \dd \widehat y.
\end{align*}
The Jacobian of $\Psi: t, \widehat x,\widehat y \mapsto t T(\widehat x,\widehat y) + (1-t) \bfx_{\edge,c_\alpha}$, which maps $[\alpha,1] \times \widehat \edge$ onto the truncated cone $C(\edge,c_\alpha)$, is given as:
\begin{align*}
    \det J_{t,\widehat x,\widehat y} \Psi = t^2 (T(\widehat x,\widehat y)- \bfx_{\edge,c_\alpha}) \cdot \left(\frac{d T(\widehat x,\widehat y)}{d \widehat x}  \wedge \frac{d T(\widehat x,\widehat y)}{d \widehat y}\right),
\end{align*}
so that for all $t \in [\alpha,1]$ and $(\widehat x,\widehat y) \in \widehat \edge$, we have by supposition:
\begin{align*}
    \left|\det J_{t,\widehat x,\widehat y} \Psi \right| \left| \frac{d T(\widehat x,\widehat y)}{d \widehat x}  \wedge \frac{d T(\widehat x,\widehat y)}{d \widehat y} \right|^{-1}
    & = t^2 \left|(T(\widehat x,\widehat y)-\bfx_{\edge,c_\alpha}) \cdot \bfn_{\edge}(T(\widehat x,\widehat y))\right| \\
    & \geq t^2 a(\edge,\bfx_{\edge,c_\alpha}) d(\edge,\bfx_{\edge,c_\alpha}),
\end{align*}
where we used that $\bfn_{\edge}(T(\widehat x,\widehat y)) = \left(\frac{d T(\widehat x,\widehat y)}{d \widehat x}  \wedge \frac{d T(\widehat x,\widehat y)}{d \widehat y}\right) \left| \frac{d T(\widehat x,\widehat y)}{d \widehat x}  \wedge \frac{d T(\widehat x,\widehat y)}{d y} \right|^{-1}$.
Since $t \geq \alpha$, we deduce:
\begin{align*}
    I_1 
    & \leq \frac{1}{(1-\alpha)} \frac{1}{\alpha^2 a(\edge,\bfx_{\edge,c_\alpha}) d(\edge,\bfx_{\edge,c_\alpha})} \norm{f}_{L^2(C(\edge,c_\alpha))}^2 .
\end{align*}
Moreover, there holds:
\begin{align*}
    | I_2 | \leq 2 D(\edge,\bfx_{\edge,c_\alpha}) \int_{\miny}^{\maxy} \int_{\minx}^{\maxx} \int_{\alpha}^{1} | \nabla_{\gamma(t)}f | | f(\gamma(t)) | \left| \frac{d T(\widehat x,\widehat y)}{d \widehat x}  \wedge \frac{d T(\widehat x,\widehat y)}{d \widehat y} \right| \dd t \dd \widehat x \dd \widehat y.
\end{align*}
Using again the bound on the Jacobian of $\Psi$, together with the Cauchy-Schwarz inequality, we obtain:
\begin{align*}
    | I_2 | 
    & \leq 2 D(\edge,\bfx_{\edge,c_\alpha}) \frac{1}{\alpha^2 a(\edge,\bfx_{\edge,c_\alpha}) d(\edge,\bfx_{\edge,c_\alpha})} \int_{C(\edge,c_\alpha)} |\nabla_\bfx f| |f(\bfx)| \dd \bfx \\
    & \leq 2 D(\edge,\bfx_{\edge,c_\alpha}) \frac{1}{\alpha^2 a(\edge,\bfx_{\edge,c_\alpha}) d(\edge,\bfx_{\edge,c_\alpha})} \norm{\nabla f}_{L^2(C(\edge,c_\alpha))} \norm{f}_{L^2(C(\edge,c_\alpha))}.
\end{align*}
We have thus shown that:
\begin{align*}
    \norm{f}_{L^2(\edge)}^2 \leq \frac{1}{\alpha^2 a(\edge,\bfx_{\edge,c_\alpha}) d(\edge,\bfx_{\edge,c_\alpha})} \left( \frac{1}{1-\alpha}\norm{f}_{L^2(C(\edge,c_\alpha))}^2 + 2 D(\edge,\bfx_{\edge,c_\alpha})  \norm{\nabla f}_{L^2(C(\edge,c_\alpha))}\norm{f}_{L^2(C(\edge,c_\alpha))} \right).
\end{align*}
Applying the inequality $ab \leq \frac{\eps a^2}{2} + \frac{b^2}{2 \eps}$, we finally conclude that:
\begin{align*}
    \norm{f}_{L^2(\edge)}^2 \leq \frac{1}{\alpha^2 a(\edge,\bfx_{\edge,c_\alpha}) d(\edge,\bfx_{\edge,c_\alpha})} \left( \frac{1 + \varepsilon }{1-\alpha}\norm{f}_{L^2(C(\edge,c_\alpha))}^2 + \frac{D(\edge,\bfx_{\edge,c_\alpha})^2}{\varepsilon}   \norm{\nabla f}_{L^2(C(\edge,c_\alpha))}^2 \right),
\end{align*}
which concludes the proof.
\end{proof}

\begin{rmrk}
    When $K$ is the image of a reference element by a smooth mapping, $T$ can be taken as the restriction of the mapping to the corresponding face of the reference element.
\end{rmrk}

To recover a similar result to the trace inequality \eqref{eq:trace}, one would need to fix the values $\alpha$ and $\varepsilon$, and bound the quantities $d(\edge,\bfx_{\edge,c_\alpha})$, $D(\edge,\bfx_{\edge,c_\alpha})$ and $a(\edge,\bfx_{\edge,c_\alpha})$.
To fix ideas, one may take $\alpha = \varepsilon = 1/2$.
In particular, the truncated cone becomes:
\begin{align*}
    C(\edge,c):=\left\{ t \bfx + (1-t) \bfx_{\edge,c}, \; \bfx \in \edge, \; t \in [1/2,1] \right\}.
\end{align*}
Furthermore, we can restrict $\bfx_{\edge,c_\alpha}$ to be in $K\backslash \edge$, so that $D(\edge,\bfx_{\edge,c_\alpha}) \leq h_K$.

Let us now discuss the mesh regularity assumptions, and the resulting bound on the quantities $d(\edge,\bfx_{\edge,c_\alpha})$ and $a(\edge,\bfx_{\edge,c_\alpha})$.
First, we require the set
\begin{align*}
    C(K,\edge) := \left\{ C(\edge,c), \; \edge \in \edges(K), \;\bfx_{\edge,c} \in K\backslash\edge, \; C(\edge,c) \subset K \right\}
\end{align*}
to be non-empty.
For regular cells with non-curved faces, this hypothesis clearly holds because $\bfx_{\edge,c}$ can always be chosen as the cell mass center.
For cells with curved faces, one may reasonably expect that taking $\bfx_{\edge,c}$ closer to the face mass center achieves the desirable result.
Furthermore, note that this assumption holds if there exist $\bfx_{\edge,c}$ such that $K$ is star-shaped with regard to $\bfx_{\edge,c}$.
However, this hypothesis is not necessary, because the full cone with base $\edge$ and apex $\bfx_{\edge,c_\alpha}$ does not need to be inside $K$, as long as $C(\edge,c_\alpha) \subset K$.
Finally, to bound $d(\edge,\bfx_{\edge,c_\alpha})$ and $a(\edge,\bfx_{\edge,c_\alpha})$, we suppose (akin to the hypothesis made in \cite{can-21-hp}), that there exists constants $\cregconeI \in \R_{>0}$ and $\cregconeII \in \R_{>0}$ such that for all $K \in \mesh$, for all $\edge \in \edges(K)$, there exists $\bfx_{\edge,c}$ in $K$ such that $C(\edge,c) \in C(K,\edge)$ and
\begin{align}
     &d(\edge,\bfx_{\edge,c})^{-1} h_K \; \leq \; \cregconeI  \label{Poisson:eq:cregconeI}, \\
     & \cregconeII \; \leq \; a(\edge,\bfx_{\edge,c}). \label{Poisson:eq:cregconeII}
\end{align}

Note that $\cregconeI$ is a bound on the anisotropy of the cells, whereas $\cregconeII$ is a bound on the curvature of the faces, as shown in Figure \ref{Poisson:fig:creg_curve}.

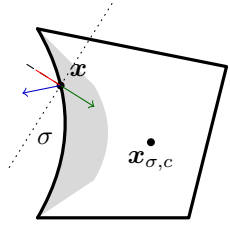
\begin{figure}[!htb]
    \centering
    \definecolor{lightred}{rgb}{0.6,0.2,0}
    \begin{tikzpicture}[scale = 0.5 ,line cap=round,line join=round,x=1.0cm,y=1.0cm,
    declare function=
    {
    a(\t) = -0.6276*\t + 3.8828;
    c(\t) = 0.2;
    b(\t) =  c(0)*\t + 3.5-c(0)*0.61;
    d(\t) = 1.59337*\t + 2.528 ;} ]
    \draw [line width=1.2pt] (0,0) -- (4,0) -- (5, 4) -- (0,5) ;
    \draw[line width=1.2pt] (0,0) to [bend right] (0,5);
    \fill (3,2) circle[radius=3pt] ;
    \draw (3,2) node[below]{ $\boldsymbol x_{\edge,c}$};
    \fill (0.61,3.5) circle[radius=3pt] ;
    \draw (0.61,3.5) node[above right ]{ $\boldsymbol x$};
    \draw[->, green!50!black] (0.61,3.5) -- (1.5,{a(1.5)}) ;
    \draw[dashed] (0.61,3.5) -- (-0.28,{a(-0.28)}) ;
    \draw[color = red ] (0.61,3.5) -- (-0.04,{a(-0.04)});
    \draw[->, blue!80!black]  (0.61,3.5) -- (-0.4,{b(-0.4)}) ;
    \draw[dotted] (0.61,3.5) -- (2,{d(2)}) ;
    \draw[dotted] (0.61,3.5) -- (-0.78,{d(-0.78)}) ;
    \fill[fill=black,fill opacity=0.15] (0,0) -- (1.5,1) to [bend right] (1.5, 3.5) -- (0,5) to [bend left] (0,0) ;
    \draw (0.2,2.5) node[below] {$\sigma$};
    \end{tikzpicture}
    \caption{
    Example of a curved face (shown here in two dimensions for clarity, although such cells are not considered in this work). 
    In grey, the truncated cone $C(\edge,c)$. 
    In green, the normalized vector $\bfx - \bfx_{\edge,c}$, in blue, the unit face normal at $\bfx$. 
    The length of the red segment is equal to the scalar product of these two vectors. 
    }
    \label{Poisson:fig:creg_curve}
\end{figure}

\bibliographystyle{abbrv}
\bibliography{bib}

\end{document}